\documentclass[a4paper,reqno]{amsart}

\usepackage{mathrsfs}
\usepackage{amsfonts}
\usepackage{amssymb}
\usepackage{amsmath}
\usepackage{graphicx}

\usepackage{amsthm}
\usepackage{amscd}
\usepackage[all,2cell]{xy}
\usepackage[usenames]{color}
\usepackage{hyperref}\usepackage{comment}
\UseAllTwocells \SilentMatrices

\numberwithin{equation}{section}

\theoremstyle{plain}
\newtheorem{thm}{Theorem}[section]
\newtheorem{prop}[thm]{Proposition}
\newtheorem{cor}[thm]{Corollary}
\newtheorem{lem}[thm]{Lemma}

\theoremstyle{definition}
\newtheorem{defi}[thm]{Definition}
\newtheorem{exm}[thm]{Example}
\theoremstyle{remark}
\newtheorem{rmk}[thm]{\bf Remark}
\newcommand{\D}{\mathbf{D}}
\newcommand{\K}{\mathbf{K}}
\newcommand{\ac}{\mathrm{ac}}

\def\Z{\mathbb{Z}}

\def\p{\partial}
\def\Ke{{\rm Ker}}
\def\Im{{\rm Im}}

\def\La{\Lambda}
\def\I{\mathcal{I}}
\def\Bb{\mathbf{B}}

\def\s{\sharp}

\def\c{\circ}
\def\d{\delta}

\def\aa{\alpha}
\def\b{\beta}
\def\z{\zeta}
\def\o{\omega}
\def\g{\gamma}
\def\bu{\bullet}
\def\lra{\longrightarrow}
\def\xra{\xrightarrow[]{}}

\newcommand{\G}{\Gamma}
\def\E{\mathcal{E}}
\renewcommand{\-}{\mbox{-}}
\newcommand{\End}{\mathrm{End}}
\newcommand{\Mod}{\mathrm{Mod}}
\newcommand{\Hom}{\mathrm{Hom}}
\newcommand{\Gr}{\mathrm{Gr}}
\newcommand{\Inj}{\mathrm{Inj}}
\newcommand{\rad}{\mathrm{rad}}

\newcommand{\shift}{\mathrm{shift}}
\renewcommand{\mod}{\mathrm{mod}}
\begin{document}

\title{The injective Leavitt complex}

\author{Huanhuan Li}



\date{\today}

\begin{abstract}For a finite quiver $Q$ without sinks, we consider the
corresponding finite dimensional algebra $A$ with radical square zero. We construct
an explicit compact generator for the homotopy category of acyclic complexes
of injective $A$-modules. We call such a generator the injective Leavitt complex of $Q$.
This terminology is justified by the following result: the differential graded endomorphism
algebra of the injective Leavitt complex of $Q$ is quasi-isomorphic to the Leavitt path algebra
of $Q$. Here, the Leavitt path algebra is naturally $\Z$-graded and viewed as a differential graded algebra with trivial
differential.
\end{abstract}

\maketitle

Let $A$ be a finite dimensional algebra over a field $k$. The homotopy category
${\mathbf{K}_{\rm ac}}(A\-\Inj)$ of acyclic complexes of injective
$A$-modules is called the stable derived category of $A$ in \cite{kr}. This category is
a compactly generated triangulated category such that its subcategory of compact objects is triangle equivalent to the singularity category \cite{bu,o}
of $A$.

In general, it seems very difficult to give an explicit compact generator for the stable derived category of an algebra. In this paper, we construct an explicit compact generator for
the homotopy category $\K_{\rm ac}(A\-\Inj)$, in the case that the algebra $A$ is with radical square zero.
The compact generator is called the \emph{injective Leavitt complex}. This terminology is justified by the following result: the differential graded endomorphism
algebra of the injective Leavitt complex is quasi-isomorphic to the Leavitt path algebra, which was introduced in \cite{ap, amp} as an algebraisation of graph $C^*$-algebras \cite{kprr,raeburn} and in particular Cuntz--Krieger algebras \cite{cuntzkrieger}. Here, the Leavitt path algebra is naturally $\Z$-graded and viewed as a differential graded algebra with trivial
differential.

Let $Q$ be a finite quiver without sinks. Set
$A=kQ/J^{2}$ to be the corresponding finite dimensional algebra
with radical square zero.
We introduce the injective Leavitt complex $\I^{\bullet}$ in Definition \ref{definj}, which is an acyclic complex of injective $A$-modules; see Proposition \ref{propacy}.

We denote by $L_{k}(Q)$ the Leavitt path algebra of $Q$ over $k$.
Recall that $L_{k}(Q)$
is a naturally $\Z$-graded algebra, which is viewed as a differential graded
algebra with trivial differential.

The main result of this paper is as follows, which combines Theorem \ref{tc} with
Theorem \ref{rightq}.

\vskip 5pt

\noindent $\mathbf{Theorem ~I.}$ \emph{Let $Q$ be a finite quiver without sinks and $A=kQ/J^{2}$ the corresponding algebra with radical square zero.}

\begin{enumerate}
\item[(1)] \emph{The injective Leavitt complex $\mathcal{I}^{\bullet}$ of $Q$ is a
compact generator for the homotopy category $\K_{\rm ac}(A\-\Inj)$.}

\item[(2)] \emph{The differential graded endomorphism algebra of the injective Leavitt complex $\I^{\bu}$ is quasi-isomorphic to the Leavitt path algebra $L_k(Q)$.\hfill $\square$}\end{enumerate}

\vskip 5pt

This result is inspired by \cite[Theorem 6.1]{cy}, which describes the homotopy category
$\K_{\rm ac}(A\-\Inj)$ in terms of Leavitt path algebras. We mention that \cite[Theorem 6.1]{cy} extends \cite[Theorem 7.2]{s}; compare \cite[Theorem 3.8]{c1}.
The construction of the injective Leavitt complex is inspired by the basis of the Leavitt path algebra given by \cite[Theorem 1]{aajz}.

For the proof of $(1)$ in Theorem~I, we construct an explicit filtration of subcomplexes of the injective Leavitt complex $\I^{\bu}$. For $(2)$, we actually prove that $\I^{\bu}$ has the structure of a
differential graded $A$-$L_k(Q)^{\rm op}$-bimodule, which is right quasi-balanced.
Here, we consider $A$ as a differential graded algebra concentrated on degree zero, and $L_k(Q)^{\rm op}$
is the opposite algebra of $L_k(Q)$, which is naturally $\Z$-graded and viewed as a differential graded algebra with trivial differential.

\begin{comment}
We denote by $\D(L_k(Q)^{\rm op})$ the derived category of left differential graded $L_k(Q)^{\rm op}$-modules. Since $\I^{\bu}$ has a  structure of a differential graded $A$-$L_k(Q)^{\rm op}$-bimodule, for any left differential graded $A$-module $M$, $\Hom_A(\I^{\bu}, M)$ has a natural structure of left differential graded $L_k(Q)^{\rm op}$-module given by the formula $(bf)(x)=(-1)^{|b|(|f|+|x|)}f(xb)$ for $b\in L_k(Q)^{\rm op}$, $f\in \Hom_A(\I^{\bu}, M)$ and $x\in \I^{\bu}$.

Applying Theorem ~I, we obtain the triangle equivalence between the homotopy category $\K_{\rm ac}(A\-\Inj)$ and the derived category $\D(L_k(Q)^{\rm op})$ which was proved in \cite[Theorem 6.1]{cy}.

\vskip 5pt

\noindent $\mathbf{Theorem ~II.}$ \emph{Let $Q$ be a finite quiver without sinks and $A=kQ/J^{2}$ the corresponding algebra with radical square zero. Then there is a triangle equivalence} $${\rm Hom}_{A}(\I^{\bullet},-):\K_{\rm ac}(A\-\Inj) \stackrel{\sim}\longrightarrow\D(L_k(Q)^{\rm op}),$$
\emph{which sends $\I^{\bu}$ to $L_k(Q)^{\rm op}$ in $\D(L_k(Q)^{\rm op})$.\hfill $\square$}

\vskip 5pt
\end{comment}

The paper is structured as follows.
In section \ref{stwo}, we introduce the
injective Leavitt complex $\I^{\bu}$ of $Q$ and prove that it is an acyclic complex of injective $A$-modules. In section \ref{sthree}, we
prove that $\I^{\bu}$ is a compact generator for
the homotopy category ${\K_{\ac}}(A\-\Inj)$.
In section \ref{sfour}, we endow $\I^{\bu}$ with a differential graded module structure over the corresponding Leavitt path algebra $L_k(Q)$.
In section \ref{sfifth}, we prove that the differential graded endomorphism algebra of $\I^{\bu}$, as a complex of $A$-modules, is quasi-isomorphic to the Leavitt path algebra.

\section{The injective Leavitt complex of a finite quiver without sinks}
\label{stwo}

In this section, we introduce the injective Leavitt complex of a finite quiver without sinks, which is an acyclic complex of injective modules over the corresponding finite dimensional algebra with radical square zero.

\subsection{The injective Leavitt complex}

Recall that a quiver $Q=(Q_{0}, Q_{1}; s, t)$
consists of a set $Q_{0}$ of vertices, a set $Q_{1}$ of arrows and
two maps $s, t: Q_{1}\xrightarrow []{}Q_{0}$,
which associate to each arrow $\alpha$ its starting vertex $s(\alpha)$ and its terminating
vertex $t(\alpha)$, respectively. A quiver $Q$
is finite if both the sets $Q_{0}$ and $Q_{1}$ are finite.

A path in the quiver $Q$
is a sequence $p=\alpha_{n}\cdots\alpha_{2}\alpha_{1}$ of arrows with
$t(\alpha_{j})=s(\alpha_{j+1})$ for $1\leq j\leq n-1$.
The length of $p$, denoted by $l(p)$, is $n$. The starting vertex of $p$, denoted by $s(p)$,
is $s(\alpha_{1})$.
The terminating vertex of $p$, denoted by $t(p)$, is $t(\alpha_{n})$.
We identify an
arrow with a path of length one. We associate to each
vertex $i\in Q_{0}$ a trivial path $e_{i}$ of length zero. Set $s(e_i)=i=t(e_i)$.
Denote by $Q_{n}$ the set of all paths in $Q$ of length $n$ for each $n\geq 0$.

Recall that a vertex of $Q$ is a sink if there is no
arrow starting at it.
For any vertex $i$ which is not a sink, fix an arrow $\gamma$
with $s(\gamma)=i$. We call the fixed arrow the \emph{special arrow} starting at $i$.
For a special arrow $\alpha$,
we set \begin{equation} \label{eq:vv}
S(\alpha)=\{\beta\in Q_{1}\;| \; s(\beta)=s(\alpha), \beta\neq\alpha\}.
\end{equation}
We mention that the terminology ``special arrow" is taken from \cite{aajz}.

The following notion is inspired by \cite[Theorem 1]{aajz}, which describes a basis for Leavitt path algebra; also see Lemma $\ref{lbasis}$.

\begin{defi}\label{da}
For two paths $p=\alpha_{m}\cdots\alpha_{1}$ and $q=\beta_{n}\cdots\beta_{1}$ in $Q$
with $m,n\geq 1$,
we call the pair $(p, q)$ an \emph{admissible pair} in $Q$ if $t(p)=t(q)$, and either
$\alpha_{m}\neq \beta_{n}$, or $\alpha_{m}=\beta_{n}$ is not special.
For each path $r$ in $Q$, we define two additional \emph{admissible pairs} $(r, e_{t(r)})$ and $(e_{t(r)}, r)$ in $Q$. \hfill $\square$
\end{defi}

For each vertex $i\in Q_0$ and $l\in \Z$, set
\begin{equation}
\mathbf{B}^{l}_{i}
=\{(p, q)\;|\;(p, q) \text{~is an admissible pair~} \text{with~} l(q)-l(p)=l\text{~and~}s(q)=i\}.\label{p}
\end{equation}

We recall that an arrow $\alpha$ in $Q$ is a loop if it satisfies $s(\alpha)=t(\alpha)$.
For a rational number $x$, denote by $\lfloor x \rfloor$ the maximal integer which is not greater than $x$.

\begin{lem} Let $Q$ be a finite quiver without sinks. The above set $\mathbf{B}^{l}_{i}$
is not empty for each vertex $i$ and each integer $l$.
\end{lem}

\begin{proof} Assume first that $l\geq 0$. Since $Q$ has no sinks,
there exists a path $q$ of length $l$
starting at $i$. Then we have $(e_{t(q)}, q)\in \mathbf{B}^{l}_{i}$.

We now consider the case $l<0$. For each vertex $i\in Q_0$, there are two subcases. The first subcase is that there exists a path $\alpha q$ such that
$s(q)=i$ and that $\alpha$ is a loop. We may assume that $q$ does not end with the loop $\aa$. Set $p$ to be the $(l(q)-l)$-th power of $\aa$. Then we have $(p, q)\in\Bb^l_i$.

The second subcase is that for the vertex $i$ there exists no path $\alpha q$ such that
$s(q)=i$ and that $\alpha$ is a loop. There exists a path
$\alpha_{n}\cdots\aa_2\alpha_{1}$ starting at $i$ such that the arrows $\aa_1,\aa_2, \cdots, \aa_{n-1}$
are pairwise distinct and that $\aa_n=\aa_m$ for some $1\leq m<n$.
We observe that $m<n-1$, since $\aa_n$ is not a loop.
Set $c=\alpha_{n-1}\cdots\aa_{m+1}\alpha_{m}$, $s=\lfloor\frac{m-l-1}{n-m} \rfloor$ and $t=m-l-1-s(n-m)$. Observe that $0\leq t<n-m$. Then we have $(c^{s}\alpha_{n-1}\cdots\alpha_{n-t}, \alpha_{m-1}\cdots\aa_2\alpha_{1})\in\Bb^l_i$. Here, if $t=0$, we understand $\alpha_{n-1}\cdots\alpha_{n-t}$ as $e_{t(\aa_{n-1})}$. If $m=1$, we understand $\alpha_{m-1}\cdots\aa_2\alpha_{1}$ as $e_i$.
\end{proof}

Let $k$ be a field and $Q$ be a finite quiver. For each $n\geq 0$, denote by $kQ_{n}$ the $k$-vector space with basis $Q_{n}$. The path algebra $kQ$ of the quiver $Q$ is defined as $kQ=\bigoplus_{n\geq 0}kQ_{n}$, whose multiplication is given as follows: for two paths $p$ and $q$, if $s(p)=t(q)$, then the product $pq$ is their concatenation;
otherwise, we set the product $pq$ to be zero. Here, we write the concatenation of paths
from right to left.

We observe that for any path $p$ and vertex $i$, $pe_{i}=\delta_{i, s(p)}p$ and
$e_ip=\delta_{i, t(p)}p$. Here, $\delta$ denotes the Kronecker symbol. It follows that the unit of
$kQ$ equals $\sum_{i\in Q_{0}}e_{i}$. Denote by $J$ the two-sided ideal of $kQ$ generated by arrows.

Consider the quotient algebra $A=kQ/J^{2}$;
it is a finite dimensional algebra with radical square zero. Indeed, $A=kQ_{0}\oplus kQ_{1}$
as a $k$-vector space and its Jacobson radical $\rad A=kQ_{1}$ satisfying $(\rad A)^{2}=0$.
For each vertex $i$ and arrow $\aa$, we still use $e_i$ and $\aa$ to denote their canonical images in $A$.

Denote by $I_{i}=D(e_{i}A)$ the injective left $A$-module for each $i\in Q_{0}$, where $(e_{i}A)_{A}$ is the
indecomposable projective right $A$-module and $D={\rm Hom}_{k}(-, k)$ denotes
the standard $k$-duality.
Denote by $\{e_{i}^{\s}\}\cup\{ \alpha^{\s}| \alpha\in Q_{1}, t(\alpha)=i\}$ the basis of
$I_i$, which is dual to the basis $\{e_{i}\}\cup\{
\alpha|\alpha\in Q_{1}, t(\alpha)=i\}$ of $e_iA$. We recall from \cite[I.2.9]{ass} that the left $A$-action on $I_i$ is as follows: for $j\in Q_{0}$ and $\beta\in Q_{1}$,
\begin{equation*}
{e_{j}}_{\cdot} e_{i}^{\s}=\delta_{i, j}e_{i}^{\s}; \;\beta_{\cdot} e_{i}^{\s}=0; \;{e_{j}}_{\cdot}\alpha^{\s}=\delta_{j, s(\alpha)}\alpha^{\s};\;\beta_{\cdot} \alpha^{\s}=\delta_{\alpha, \beta} e_{i}^{\s}.
\end{equation*}

We have the following direct observation.

\begin{lem} \label{a} Let $i, j$ be two vertices in $Q$, and $f: I_i\xrightarrow[]{} I_j$ be a $k$-linear map. Then $f$ is a left $A$-module morphism if and only if
$f(e_i^{\s})=\delta_{i, j}\lambda e_j^{\s}$ and $f(\aa^{\s})=\delta_{i, j}\lambda\aa^{\s}+
\delta_{j, s(\aa)}\mu(\aa) e_{j}^{\s}$
with $\lambda$ and $\mu(\aa)$ scalars for all $\alpha\in Q_1$
with $t(\aa)=i$.  \hfill $\square$
\end{lem}

For a set $X$ and an $A$-module $M$, the coproduct $M^{(X)}$ will be
understood as
$\bigoplus_{x\in X}M\zeta_{x}$, where each component $M\zeta_{x}$ is $M$.
For an element $m\in M$, we use $m\zeta_{x}$
to denote the corresponding element in $M\zeta_{x}$.

For a path $p=\alpha_{n}\cdots\aa_2\alpha_{1}$ in $Q$ of length
$n\geq 2$, we denote by $\widehat{p}=\alpha_{n-1}\cdots\alpha_{1}$ and $\widetilde{p}=\alpha_{n}\cdots\alpha_{2}$ the two \emph{truncations} of
$p$. For an arrow $\aa$, denote by $\widehat{\aa}=e_{s(\aa)}$ and $\widetilde{\aa}=e_{t(\aa)}$.

\begin{defi} \label{definj} Let $Q$ be a finite quiver without sinks. The \emph{injective Leavitt complex} $\I^{\bullet}=(\I^{l}, \partial^{l})_{l\in \Z}$ of $Q$ is defined as follows:
\begin{enumerate}\item[(1)] the $l$-th component $\I^{l}=\bigoplus_{i\in Q_{0}}{I_{i}}^{(\Bb^l_i)}$;

\item[(2)] the differential $\partial^{l}:\I^{l}\xrightarrow[]{} \I^{l+1}$ is given by
$\partial^{l}(e_{i}^{\s}\zeta_{(p, q)})=0$ and
\begin{equation*}\partial^{l}(\alpha^{\s}\zeta_{(p, q)})=\begin{cases}
 e_{s(\alpha)}^{\s}\zeta_{(\widehat{p}, e_{s(\alpha)})}-\sum\limits_{\beta\in S(\alpha)} e_{s(\alpha)}^{\s}\zeta_{(\beta\widehat{p}, \beta)}, &
 \begin{matrix}\text{if~} q=e_{i}, ~~p=\alpha \widehat{p} \\\text{~and~} \alpha \text{~is special};\end{matrix}\\
 e_{s(\alpha)}^{\s}\zeta_{(p, q\alpha)},& \text{otherwise},
\end{cases}
\end{equation*}
for any $i\in Q_{0}$, $(p, q)\in \Bb^l_i$ and $\alpha\in Q_{1}$ with $t(\alpha)=i$. Here, the set $S(\alpha)$ is defined in $(\ref{eq:vv})$. \hfill $\square$
\end{enumerate}
\end{defi}

We observe that in the first case of defining $\partial^{l}(\alpha^{\s}\zeta_{(p, q)})$, we have that $l<0$ and that the pair $(p, q\aa)=(\aa \widehat{p}, \aa)$ is not admissible.

Each component $\I^l$ is an injective $A$-module. The differentials $\p^l$
are $A$-module morphisms; compare Lemma \ref{a}. It is direct to see
that $\p^{l+1}\circ\p^l$=0 for each $l\in\Z$. In summary, $\I^{\bu}$
is a complex of injective $A$-modules. 
%For a picture of $\I^{\bu}$, see \ref{}.

The following fact is immediate.

\begin{lem} \label{partial}For each $i\in Q_{0}$, $l\in\Z$ and $(p, q)\in\Bb^l_i$,
the following statements hold.

\begin{enumerate}\item[(1)] If $q=e_{i}$, then $e_{i}^{\s}\zeta_{(p, e_{i})}=\sum_{\{\alpha\in Q_{1}\;|\; s(\alpha)=i\}}\partial^{l-1}(
\alpha^{\s}\zeta_{(\alpha p, e_{t(\alpha)})})$.

\item[(2)] If $q=\widetilde{q}\alpha$ with $\widetilde{q}$ the truncation of $q$, then
$e_{i}^{\s}\zeta_{(p, q)}=\partial^{l-1}(\alpha^{\s}\zeta_{(p, \widetilde{q})}).$
\hfill $\square$
\end{enumerate}
\end{lem}

\subsection{The acyclicity of the injective Leavitt complex}
We will show that the injective Leavitt complex is acyclic. We observe some lemmas on linear maps. For the convenience of the reader, we give full proofs.

In what follows,
$f: V\xrightarrow []{}V'$ is a $k$-linear map between two vector spaces $V$ and $V'$.
Suppose that $B$ and $B'$ are $k$-bases of $V$ and $V'$, respectively.

We say that
the triple $(f, B, B')$
satisfies \emph{Condition $\rm{(X)}$} if $f(B)\subseteq B'$ and the restriction of $f$ on $B$ is injective. The following observation is immediate.

\begin{lem} \label{zero} Assume that the triple $(f, B, B')$ satisfies Condition $\rm{(X)}$.
Then we have $\Ke f=0$. \hfill $\square$
\end{lem}

We suppose further that there are disjoint unions $B=B_{0}\cup B_{1}\cup B_{2}$
and $B'=B'_{0}\cup B'_{1}$.
We say that the triple $(f, B, B')$
satisfies \emph{Condition \rm{(Y)}} if
the following statements hold:

\begin{enumerate}\item[(Y1)] $f(b)=0$ for each $b\in B_{0}$;

\item[(Y2)] $f(B_{1})\subseteq B_{1}'$ and $(f_{1}, B_1, B')$ satisfies Condition $\rm{(X)}$, where $f_{1}$ is the restriction of $f$ to the subspace spanned by $B_{1}$;

\item[(Y3)]  for each $b\in B_{2}$, we have 
\begin{equation} \label{condition3}f(b)=b_{0}-\sum_{y\in B_{1}(b)}f(y)
\end{equation} for some $b_{0}\in B_{0}'$ and some finite subset $B_1(b)\subseteq B_1$. Moreover, if $b\neq c$ in $B_2$, we have $b_{0}\neq c_{0}$.
\end{enumerate}

\begin{lem} \label{kerim1} Assume that $(f, B, B')$ satisfies Condition $\rm{(Y)}$.
Then $B_{0}$ is a $k$-basis of $\Ke f$ and $f(B_{1})\cup \{b_{0}\;|\; b\in B_2\}$ is a $k$-basis of $\Im f$.
\end{lem}
\begin{proof} Let $x=\sum_{c\in B_{1}}\mu_{c}c+\sum_{b\in B_{2}}\nu_{b}b$
with $\mu_{c}$ and $\nu_{b}$ scalars satisfy $f(x)=0$. We claim that $x=0$.
In view of (Y1), the claim implies that $B_{0}$ is a $k$-basis of $\Ke f$.

We consider
\begin{align*}f(x)&=\sum_{c\in B_{1}}\mu_{c}f(c)+\sum_{b\in B_{2}}\nu_{b}(b_{0}-\sum_{y\in B_{1}(b)}f(y))\\
&=\sum_{b\in B_{2}}\nu_{b}b_{0}+(\sum_{c\in B_{1}}\mu_{c}f(c)-\sum_{b\in B_{2}}\sum_{y\in B_{1}(b)}\nu_b f(y))=0.\end{align*}
We notice that $b_0\in B_0'$, $f(c)\in B_1'$ and $f(y)\in B_1'$.
So we have $\sum_{b\in B_{2}}\nu_{b}b_{0}=0$. By (Y3), we infer that $\nu_{b}=0$ for each
$b\in B_2$. Then we observe that $\sum_{c\in B_{1}}\mu_{c}f(c)=0$.
We apply (Y2) and Lemma \ref{zero} for the triple $(f_1, B_1, B')$ to obtain $\sum_{c\in B_{1}}\mu_{c}c=0$,
proving that $x=0$.

Observe that $b_0$ belongs to $\Im f$ for each $b\in B_2$. By (Y1) and (Y3),
each element in $\Im f$ is a $k$-linear combination of elements from $f(B_1)\cup\{b_{0}| b\in B_2\}$. Recall that $f(B_1)\subseteq B_1'$ and $\{b_{0}| b\in B_2\}\subseteq B_0'$.
This proves that the set
$f(B_{1})\cup \{b_{0}| b\in B_2\}$ is a $k$-basis of $\Im f$.
\end{proof}

Let $f: V\xrightarrow[]{} V'$ be a $k$-linear map with $B$ and $B'$ bases of $V$ and $V'$,
respectively.
Suppose that there is a disjoint union $B=B_{0}\cup B_{1}\cup B_{2}$.
We say that the triple $(f, B, B')$
satisfies \emph{Condition $\rm{(Z)}$} if
the following statements hold:
\begin{enumerate}
\item[(Z1)] $f(b)=0$ for each $b\in B_{0}$;

\item[(Z2)] $f(B_{1})\subseteq B'$ and $(f_{1}, B_1, B')$ satisfies Condition $\rm{(X)}$, where $f_{1}$ is the restriction of $f$ to the subspace spanned by $B_{1}$;

\item[(Z3)] for each $b\in B_{2}$, we have $f(b)=-\sum_{y\in B_{1}(b)}f(y)$ for some
finite subset $B_1(b)\subseteq B_1$.
\end{enumerate}

\begin{lem} \label{kerim2} Assume that $(f, B, B')$ satisfies Condition $\rm{(Z)}$.
Then we have that $B_{0}\cup \{b+\sum_{y\in B_{1}(b)}y\;|\; b\in B_2\}$ is a $k$-basis of $\Ke f$ and that $f(B_{1})$ is a $k$-basis of $\Im f$.
\end{lem}
\begin{proof} By (Z3), the element $b+\sum_{y\in B_{1}(b)}y$ lies in $\Ke f$ for each $b\in B_2$.
We observe that the set $B_{0}\cup \{b+\sum_{y\in B_{1}(b)}y\;|\; b\in B_2\}$ is linearly independent.
Let $x=\sum_{c\in B_{1}}\mu_{c}c+\sum_{b\in B_{2}}\nu_{b}b$ with
$\mu_{c}$ and $\nu_{b}$ scalars satisfy $f(x)=0$. We claim that
$x$ belongs to the subspace spanned by $\{b+\sum_{y\in B_{1}(b)}y\;|\; b\in B_2\}$.
By (Z1), this claim implies that $B_{0}\cup \{b+\sum_{y\in B_{1}(b)}y\;| \;b\in B_2\}$ is a $k$-basis of $\Ke f$.

We write $x=x'+\sum_{b\in B_{2}}\nu_{b}(b+\sum_{y\in B_{1}(b)}y)$. Then $x'$
belongs to the subspace spanned by $B_1$. Observe that $f(x')=0$. Applying (Z2) and Lemma \ref{zero} for the triple $(f_1, B_1, B')$, we infer that $x'=0$,
proving that $x=\sum_{b\in B_{2}}\nu_{b}(b+\sum_{y\in B_{1}(b)}y)$.

By (Z1) and (Z3), any element of $\Im f$ is a $k$-linear combination of elements from $f(B_1)$. Since $f(B_1)\subseteq B'$, it follows that $f(B_{1})$ is a $k$-basis of $\Im f$.
\end{proof}

In what follows, $Q$ is a finite quiver without sinks. We
consider the differential $\p^l:\I^l\xrightarrow[]{}\I^{l+1}$ in Definition
\ref{definj}. We have the following $k$-basis of $\I^l$:
$$\G^l=\{e_i^{\s}\z_{(p, q)}, \aa^{\s}\z_{(p, q)}\;|\;i\in Q_0, (p, q)\in \Bb^l_i \;\text{and}\; \aa\in Q_1 \;\text{with} \;t(\aa)=i\}.$$
Consider a subset $\G^{l}_{0}=\{e_{i}^{\s}\zeta_{(p, q)}\;|\;i\in Q_{0}\text{~and~} (p, q)\in \Bb^l_i\}$ of $\G^l$. If $l<0$, we put $\G^{l}_{2}=\{\alpha^{\s}\zeta_{(p, e_{t(\aa)})}\;|\;\alpha  \text{~is special and~} p=\alpha \widehat{p}\in Q_{-l}\}$; if $l\geq 0$, we put $\G^{l}_{2}=\emptyset$. Take $\G^{l}_{1}=\G^{l}\setminus(\G^{l}_{0}\cup \G^{l}_{2})$. Then we have a disjoint union
$\G^l=\G^{l}_{0}\cup\G^{l}_{1}\cup\G^{l}_{2}$.

Observe that $\G^{l+1}$ is a $k$-basis of $\I^{l+1}$.
Consider a subset $\G'^{l+1}_0=\{e_{i}^{\s}\zeta_{(p, e_{i})}\;|\;i\in Q_{0}, (p, e_{i})\in \mathbf{B}_{i}^{l+1}\}$ of $\G^{l+1}$. We mention that $\G'^{l+1}_0=\emptyset$ for $l\geq 0$.
Set $\G'^{l+1}_1=\G^{l+1}\setminus\G'^{l+1}_0$. Then
we have a disjoint union $\G^{l+1}=\G'^{l+1}_0\cup\G'^{l+1}_1$.

Recall that a complex $X^{\bullet}$ of $A$-modules is \emph{acyclic} if the $l$-th cohomology $H^l(X^{\bullet})=0$ for each $l\in\Z$. The main result of this section is as follows.

\begin{prop} \label{propacy}Let $Q$ be a finite quiver without sinks. Then the injective Leavitt complex $\I^{\bullet}$
of $Q$ is an acyclic complex.
\end{prop}

\begin{proof} The statement follows immediately from Lemma \ref{partialker}.
 \end{proof}

\begin{lem} \label{partialker} For each $l\in\mathbb{Z}$, the set $\G^{l}_{0}$
is a $k$-basis of ${\Ke}\p^{l}$ and the set $\G^{l+1}_0$ is a $k$-basis of $\Im \p^{l}$.
\end{lem}
\begin{proof} We observe that the triple $(\p^l, \G^l, \G^{l+1})$
satisfies Condition (Y). Indeed, for any $b\in\G^l_0$, $\p^l(b)=0$. The differential $\p^l$ induces an injective map $\p^l:\G^l_1\xra \G'^{l+1}_1$. Then we have (Y1) and (Y2). To see (Y3), for $l<0$ and each element $\aa^{\s}\z_{(p, e_{t(\aa)})}\in \G^l_2$, we have $p=\aa\widehat{p}$ and that $\aa$ is special. By the differential $\p^l$ in Definition \ref{definj}, we have 
\begin{equation*}
\begin{split}
\p^l(\aa^{\s}\z_{(p, e_{t(\aa)})})
&=e_{s(\alpha)}^{\s}\zeta_{(\widehat{p}, e_{s(\alpha)})}-\sum\limits_{\beta\in S(\alpha)} e_{s(\alpha)}^{\s}\zeta_{(\beta\widehat{p}, \beta)}\\
&=e_{s(\alpha)}^{\s}\zeta_{(\widehat{p}, e_{s(\alpha)})}-\sum\limits_{\beta\in S(\alpha)} \p^l(\beta^{\s}\zeta_{(\beta\widehat{p}, e_{t(\beta)})}).
\end{split}
\end{equation*} By \eqref{condition3}, $(\aa^{\s}\z_{(p, e_{t(\aa)})})_0=e_{s(\aa)}^{\s}\z_{(\widehat{p}, e_{s(\aa)})}$
and $\G^l_1(\aa^{\s}\z_{(p, e_{t(\aa)})})=\{\b^{\s}\z_{(\b\widehat{p}, e_{t(\b)})}\;|\;\b\in S(\aa)\}$.

We recall that $\G^{l+1}_0=\{e_{i}^{\s}\zeta_{(p, q)}\;|\;i\in Q_{0}\text{~and~} (p, q)\in \Bb^{l+1}_i\}$.
We observe that $\p^l(\G^{l}_{1})\cup \{e_{s(\aa)}^{\s}\z_{(\widehat{p}, e_{s(\aa)})}\;|\;\aa^{\s}\z_{(p, e_{t(\aa)})}\in \G^{l}_2\}=\G^{l+1}_0$.
Indeed, if $l(q)\geq 1$, we apply Lemma $\ref{partial}$(2);
if $q=e_i$, let $\aa$ be the special arrow starting at $i$ and
set $p'=\aa p$. Then we have $e_{s(\aa)}^{\s}\z_{(\widehat{p'}, e_{s(\aa)})}
=e_i^{\s}\z_{(p, q)}$. The statements follow by applying Lemma $\ref{kerim1}$
for the triple $(\p^l, \G^l, \G^{l+1})$.
\end{proof}

\begin{exm} \label{oneandone}Let $Q$ be the following quiver with one vertex and one loop. $$\xymatrix{\scriptstyle{1}\cdot\ar@(ur, dr)[]|{\alpha}}$$
The unique arrow $\aa$ is special. Set $e=e_1$ and $\mathbf{B}^{l}=\Bb^l_1$
for each $l\in \mathbb{Z}$.
It follows that
\begin{equation*}\Bb^l=\begin{cases}
 \{(\aa^{-l}, e)\}, &
 \text{if $l<0$};\\
  \{(e, e)\},& \text{if $l=0$};\\
   \{(e, \aa^{l})\},&\text{if $l>0$}.
\end{cases}
\end{equation*}

The corresponding algebra $A$ with radical square zero
is isomorphic to $k[x]/(x^2)$.
Set $I=D(A_A)$. Write $I^{(\Bb^l)}=I\z^l$, where $\z^l=\z_{(\aa^{-l}, e)}$ for
$l<0$, $\z^0=\z_{(e, e)}$ and $\z^l=\z_{(e, \aa^l)}$ for $l>0$.
Then the injective Leavitt complex $\I^{\bullet}$ of $Q$ is as follows
$$
\cdots
\stackrel{}{\longrightarrow} I\z^{l-1}
\stackrel{\p^{l-1}}{\longrightarrow} I\z^{l}
\stackrel{\p^{l}}{\longrightarrow} I\z^{l+1}
\stackrel{}{\longrightarrow} \cdots,
$$ where the differential $\p^l$ is given by $\p^{l}(e^{\s}\zeta^l)=0$
and $\p^{l}(\aa^{\s}\zeta^l)=e^{\s}\z^{l+1}$ for each $l\in \Z$.

Observe that $A$ is a self-injective algebra with a unique simple module. The injective Leavitt complex $\I^{\bu}$
is isomorphic to a complete resolution of the simple $A$-module; see \cite[Definition 3.1.1]{bu} and compare \cite[Proposition 2.20]{th}.
\end{exm}

\begin{exm} \label{exm} Let $Q$ be the following quiver with one vertex and two loops. $$
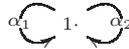
\xymatrix{\scriptstyle{1}\cdot \ar@(ul, dl)[]|{\alpha_{1}} \ar@(ur, dr)[]|{\alpha_{2}}}$$
We choose $\alpha_{1}$ to be the special arrow.
Set $e=e_1$ and $\mathbf{B}^{l}=\Bb^l_1$
for each $l\in \mathbb{Z}$.
A pair $(p,q)$ of paths lies in $\Bb^l$ if and only if
$l(q)-l(p)=l$ and $p, q$ do not end with $\aa_1$ simultaneously.
In particular, the set $\Bb^l$ is infinite.

The corresponding algebra $A$ with radical square zero
has a $k$-basis $\{e, \aa_1, \aa_2\}$.
Set $I=D(A_{A})$.
Then the injective Leavitt complex $\I^{\bullet}$ of $Q$ is as follows.
$$
\cdots
\stackrel{}{\longrightarrow} I^{({\mathbf{B}^{-1}})}
\stackrel{\p^{-1}}{\longrightarrow} I^{({\mathbf{B}^{0}})}
\stackrel{\p^{0}}{\longrightarrow} I^{({\mathbf{B}^{1}})}
\stackrel{}{\longrightarrow} \cdots
$$ We write the differential $\p^{-1}$ explicitly: $\p^{-1}(e^{\s}\zeta_{(p, q)})=0$, $\p^{-1}(\alpha_{2}^{\s} \zeta_{(p, q)})=e^{\s}\zeta_{(p, q\alpha_{2})}$ and

\begin{equation*}\p^{-1}(\alpha_{1}^{\s}\zeta_{(p, q)})=\begin{cases}
 e^{\s}\zeta_{(e, e)}-e^{\s}\zeta_{(\alpha_{2}, \alpha_{2})},  &\text{if $q=e$ and $p=\alpha_{1}$};\\
 e^{\s}\zeta_{(p, q\alpha_{1})},& \text{otherwise},
\end{cases}
\end{equation*} for $(p, q)\in \mathbf{B}^{-1}$.
\end{exm}

\section{The injective Leavitt complex as a compact generator}
\label{sthree}

In this section, we prove that the injective Leavitt complex of a finite quiver without sinks is a compact generator in
the homotopy category of acyclic complexes of injective modules over the corresponding
algebra with radical square zero.

\subsection{The cokernel complex and its subcomplexes}
\label{subsection31}

Let $Q$ be a finite quiver without sinks and $A=kQ/J^2$ be the corresponding algebra with radical square zero.
For each $i\in Q_0$, $l\in \Z$ and $n\geq 0$,
denote by
$$\mathbf{B}^{l, n}_{i}=\{(p, q)\;|\; (p, q)\in \mathbf{B}^{l}_{i} \text{~with~} q\in Q_{n} \}.$$ We refer to $(\ref{p})$ for the definition of $\Bb^l_i$. In addition, we set $\Bb^{l, n}_i=\emptyset$ if $n<0$.

Recall the injective Leavitt complex $\I^{\bu}$ of $Q$. For $l\in\Z$, $n\geq 0$, we simply denote $\bigoplus_{i\in Q_0}I_i^{(\Bb^{l, n}_i)}$ by $I^{(\Bb^{l,n})}$. We visually represent the injective Leavitt complex $\I^{\bu}$. 

\bigskip

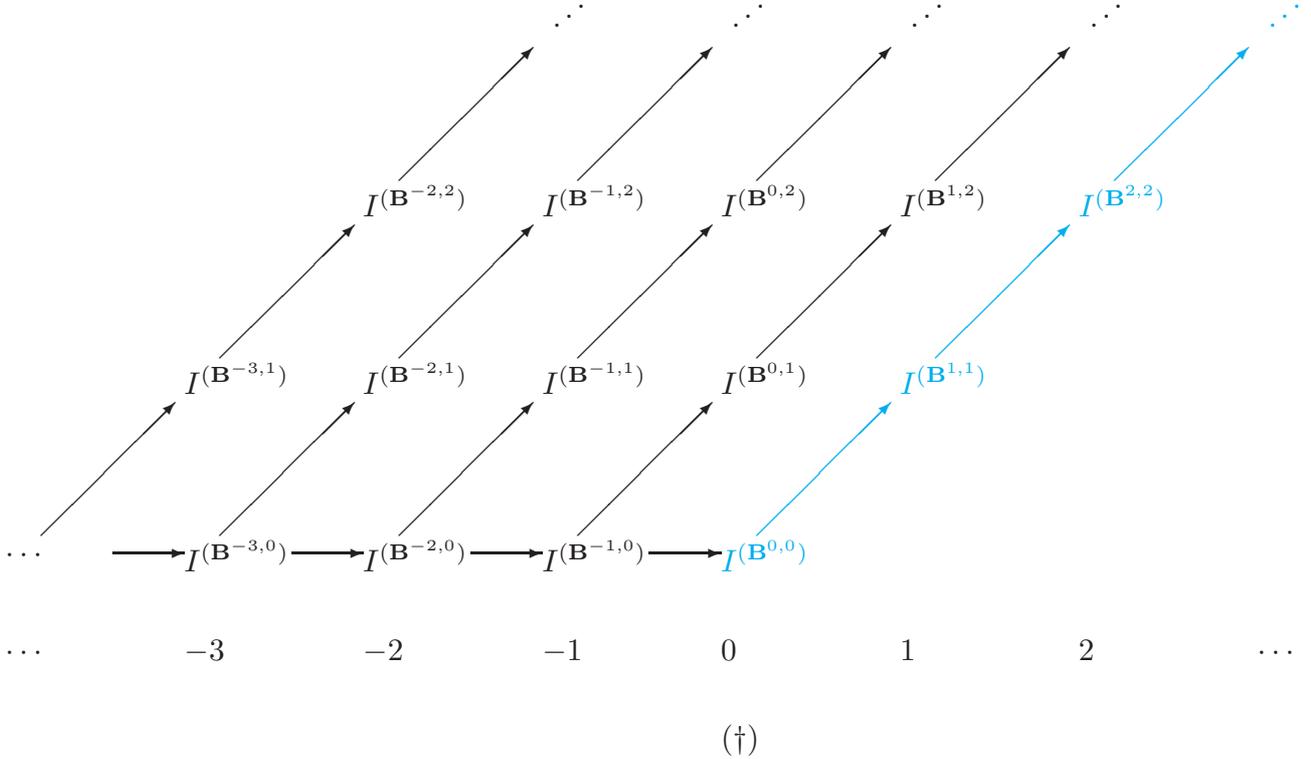
\[\resizebox{1.\hsize}{!}{$
$$~~~~~~~~~~~~~~~~~~~~~~~\setlength{\unitlength}{1mm}
\begin{picture}(150,90)

\put(120,60){\textcolor{cyan}{$I^{(\Bb^{2, 2})}$}}
\put(100,60){$I^{(\Bb^{1, 2})}$} 
\put(80,60){$I^{(\Bb^{0, 2})}$}\put(60,60){$I^{(\Bb^{-1, 2})}$}
\put(40,60){$I^{(\Bb^{-2, 2})}$}
%\put(40,60){$\cdots$}
%\qbezier[3](40, 60)(40,60)(44,64)

\put(61,81){\reflectbox{$\ddots$}}
\put(81,81){\reflectbox{$\ddots$}}
\put(101,81){\reflectbox{$\ddots$}}
\put(121,81){\reflectbox{$\ddots$}}
\put(141,81){\textcolor{cyan}{\reflectbox{$\ddots$}}}
%\qbezier[6](60, 80)(60,80)(64,84)
%\qbezier[6](80, 80)(80,80)(84,84)
%\qbezier[6](100, 80)(100,80)(104,84)
%\qbezier[6](120, 80)(120,80)(124,84)
%\qbezier[6](140, 80)(140,80)(144,84)

\put(44,64){\vector (1,1){15}}
\put(64,64){\vector (1,1){15}}
\put(84,64){\vector (1,1){15}}
\put(104,64){\vector (1,1){15}}
\put(124,64){\textcolor{cyan}{\vector (1,1){15}}}

\put(100,40){\textcolor{cyan}{$I^{(\Bb^{1, 1})}$}}
\put(80,40){$I^{(\Bb^{0, 1})}$}\put(60,40){$I^{(\Bb^{-1, 1})}$}
\put(40,40){$I^{(\Bb^{-2, 1})}$}
\put(20,40){$I^{(\Bb^{-3, 1})}$}  
%%\put(20,40){$\cdots$}
%\qbezier[3](20, 40)(20,40)(24,44)

\put(24,44){\vector (1,1){15}}
\put(44,44){\vector (1,1){15}}
\put(64,44){\vector (1,1){15}}
\put(84,44){\vector (1,1){15}}
\put(104,44){\textcolor{cyan}{\vector (1,1){15}}}

\put(80,20){\textcolor{cyan}{$I^{(\Bb^{0, 0})}$}}
\put(60,20){$I^{(\Bb^{-1, 0})}$}
\put(40,20){$I^{(\Bb^{-2, 0})}$} \put(20,20){$I^{(\Bb^{-3, 0})}$}
\put(0,21){$\cdots$}

\put(84,24){\textcolor{cyan}{\vector (1,1){15}}}
\put(64,24){\vector (1,1){15}}
\put(44,24){\vector (1,1){15}}
\put(24,24){\vector (1,1){15}}
\put(4,24){\vector (1,1){15}}

\put(72,22){\vector (1,0){8}}
\put(52,22){\vector (1,0){8}}
\put(32,22){\vector (1,0){8}}
\put(12,22){\vector (1,0){8}}
%\put(-9,22){\vector (1,0){8}}
%\put(-14,21){$\cdots$}

\put(120, 10){$2$}\put(100, 10){$1$}
\put(80, 10){$0$}\put(60, 10){$-1$}\put(40, 10){$-2$}\put(20,10){$-3$}
%\put(0, 10){$-4$}
\put(0, 10){$\cdots$}\put(140,10){$\cdots$}
\put(80,0){$\mathbf{(\dag)}$}
\end{picture}$$ $}\]

\bigskip

\begin{rmk} 
 For each $l\in\Z$, the $l$-th component of the injective Leavitt complex $\I^{\bu}$ is the coproduct of the objects in the $l$-th column of the above diagram. The differentials of $\I^{\bu}$ are combinations of the maps in the diagram.
\end{rmk}

We will give a subcomplex of the injective Leavitt complex $\I^{\bu}$.
For each $l\geq 0$, consider the submodule $M^l=\bigoplus_{i\in Q_{0}}I_{i}^{(\mathbf{B}^{l, l}_{i})}\subseteq \I^l$, where
$I_i=D(e_iA)$. Observe that the differential $\p^l:\I^l\xrightarrow[]{}\I^{l+1}$
satisfies $\p^l(M^l)\subseteq M^{l+1}$.
Then we have a subcomplex $M^{\bu}$ of $\I^{\bu}$,
by setting $M^l=0$ for $l<0$.
Denote by $\iota^{\bullet}=(\iota^{l})_{l\in \Z}: M^{\bullet}\xrightarrow[]{} \I^{\bullet}$ the inclusion chain map with $\iota^l=0$ for $l<0$. We set $C^{\bu}$ to be the cokernel of $\iota^{\bu}$, called the \emph{cokernel complex} in this section. The objects on the first cyan coloured diagonal line on the right of the diagram $(\dag)$ form the subcomplex $M^{\bu}$, while the other part gives rise to the cokernel complex $C^{\bu}$.

We now describe the cokernel complex $C^{\bu}=(C^l, \widetilde{\p}^l)_{l\in\Z}$. For each vertex $i\in Q_{0}$ and $l\in\Z$,
set \begin{equation}\label{integer}
{\mathbf{B}^{l, +}_{i}}=\bigcup_{m>l}\mathbf{B}^{l, m}_{i}.
\end{equation}
We have the disjoint union $\Bb^l_i=\Bb^{l,l}_i\cup {\mathbf{B}^{l, +}_{i}}$
for each $l\geq 0$, and $\Bb^l_i={\mathbf{B}^{l, +}_{i}}$ for $l<0$. The component of $C^{\bu}$ is $C^l=\bigoplus_{i\in Q_{0}}I_{i}^{({\mathbf{B}^{l, +}_{i}})}$ for each $l\in \Z$. We have $C^l=\I^l$ for $l<0$ and $\widetilde{\p}^l=\p^l$ for $l\leq -2$. For $l\geq 0$, we have $C^l\subseteq \I^l$ and $\p^l(C^l)\subseteq C^{l+1}$; indeed, $\I^l=M^l\oplus C^l$. Thus the differential $\widetilde{\p}^l$ on $C^l$ is the restriction of $\p^l$ to $C^l$ for $l\geq 0$. In summary,
the complex $C^{\bu}$ is as follows
$$\cdots\stackrel{}{\longrightarrow}
C^{-3}\stackrel{\p^{-3}}{\longrightarrow}
C^{-2}\stackrel{\p^{-2}}{\longrightarrow}
C^{-1}\stackrel{\widetilde{\partial}^{-1}}{\longrightarrow}
C^{0}\stackrel{\p^0}{\longrightarrow}
C^{1}\stackrel{\p^{1}}{\longrightarrow}
C^2\stackrel{}{\lra}
\cdots,
$$
where the differential $\widetilde{\partial}^{-1}$ is given such that  $\widetilde{{\partial}}^{-1}(e_{i}^{\s}\zeta_{(p, q)})=0$ and
\begin{equation*}\widetilde{{\partial}}^{-1}(\alpha^{\s}\zeta_{(p, q)})=\begin{cases}
-\sum\limits_{\beta\in S(\alpha)} e_{s(\alpha)}^{\s}\zeta_{(\beta, \beta)}, &\text{if $p=\alpha$ and $\alpha$ is special;}\\
 e_{s(\alpha)}^{\s}\zeta_{(p, q\alpha)},& \text{otherwise},
\end{cases}
\end{equation*} for any $i\in Q_{0}$, $(p, q)\in \mathbf{B}^{-1}_{i}$ and $\alpha\in Q_{1}$ with $t(\alpha)=i$.
Here, we recall from $(\ref{eq:vv})$ the definition of $S(\alpha)$.
We emphasize that the differential $\widetilde{\p}^{-1}$
is induced by the differential $\p^{-1}$ in Definition \ref{definj}.

We observe the following inclusions inside the complex $C^{\bu}$:
\begin{equation}\label{jsj1}\p^{l}(\bigoplus_{i\in Q_{0}}I_{i}^{(\mathbf{B}^{l, 0}_{i})})\subseteq
(\bigoplus_{i\in Q_{0}}I_{i}^{(\mathbf{B}^{l+1, 0}_{i})})\oplus(\bigoplus_{i\in Q_{0}}I_{i}^{(\mathbf{B}^{l+1, 1}_{i})})\text{~for each~} l\leq -2,\end{equation}
\begin{equation}\label{jsj2}\widetilde{\p}^{-1}(\bigoplus_{i\in Q_{0}}I_{i}^{(\mathbf{B}^{-1, 0}_{i})})\subseteq
\bigoplus_{i\in Q_{0}}I_{i}^{(\mathbf{B}^{0, 1}_{i})},\end{equation}
\begin{equation}\label{jsj3}\text{~~and\quad} \p^l(\bigoplus_{i\in Q_{0}}I_{i}^{(\mathbf{B}^{l, n}_{i})})\subseteq \bigoplus_{i\in Q_{0}}I_{i}^{(\mathbf{B}^{l+1, n+1}_{i})} \text{~for each~} l\in\Z \text{~and~} n>0.\end{equation}
By (\ref{jsj2}) and (\ref{jsj3}), the following complex $C_1^{\bu}$
is a subcomplex of $C^{\bu}$ satisfying $C^l_1=0$ for $l<-1$
$$0\rightarrow\bigoplus_{i\in Q_{0}}I_{i}^{(\mathbf{B}^{-1, 0}_{i})}
\stackrel{\widetilde{\partial}^{-1}}{\longrightarrow} \bigoplus_{i\in Q_{0}}I_{i}^{(\mathbf{B}^{0, 1}_{i})}\stackrel{\partial^{0}}{\longrightarrow}
 \bigoplus_{i\in Q_{0}}I_{i}^{(\mathbf{B}^{1, 2}_{i})}\stackrel{}{\longrightarrow}\cdots$$
$$\cdots\stackrel{}{\longrightarrow} \bigoplus_{i\in Q_{0}}I_{i}^{(\mathbf{B}^{l, l+1}_{i})}
\stackrel{\partial^{l}}{\longrightarrow} \bigoplus_{i\in Q_{0}}I_{i}^{(\mathbf{B}^{l+1, l+2}_{i})}
\stackrel{}{\longrightarrow} \cdots.$$
In general, by (\ref{jsj1}), (\ref{jsj2}) and (\ref{jsj3}), we have the subcomplex $C_{n}^{\bullet}$ of $C^{\bu}$ for each $n\geq 2$, whose components are given by
\begin{equation*}C^l_n=
\begin{cases}
0, &\text{if $l<-n$};\\ \bigoplus_{i\in Q_0}I_i^{(\Bb^{l, 0}_i)}, & \text{if $l=-n$;}\\
(\bigoplus_{i\in Q_0}I_i^{(\Bb^{l, 0}_i)})
\oplus\cdots\oplus(\bigoplus_{i\in Q_0}I_i^{(\Bb^{l, l+n}_i)}), & \text{if $-n<l<0$;}\\
 (\bigoplus_{i\in Q_0}I_i^{(\Bb^{l, l+1}_i)})\oplus\cdots\oplus(\bigoplus_{i\in Q_0}I_i^{(\Bb^{l, l+n}_i)})
, & \text{if $l\geq 0$.}
\end{cases}
\end{equation*}
For each $n\geq 1$, the first nonzero component of the complex $C_n^{\bu}$ is
$C^{-n}_n$. 

\begin{rmk}
The subcomplexes $C^{\bu}_n$ of the cokernel complex $C^{\bu}$ can be visually seen from the diagram $(\dag)$. For instance, the subcomplex $C_1^{\bu}$ consists of objects on the first black diagonal line on the right of the diagram $(\dag)$, the subcomplex $C_2^{\bu}$ consists of objects on the first two black diagonal lines on the right of the diagram $(\dag)$, and so on.
\end{rmk}

We denote by $i_n^{\bullet}: C_{n}^{\bullet}\xrightarrow [] {}C_{n+1}^{\bullet}$ the inclusion chain map for each $n\geq 1$.
We obtain the sequence
\begin{equation*}
C_{1}^{\bullet}\stackrel{i_{1}^{\bullet}}{\longrightarrow}C_{2}^{\bullet}\stackrel{i_{2}^{\bullet}}
{\longrightarrow}C_3^{\bu} \stackrel{}{\longrightarrow}\cdots \stackrel{}{\longrightarrow}C_{n}^{\bullet}\stackrel{i_{n}^{\bullet}}
{\longrightarrow}C_{n+1}^{\bullet}\stackrel{}{\longrightarrow}\cdots.
\end{equation*}
We observe that $C^{\bullet}=\lim\limits_{\xrightarrow[] {}} C_{n}^{\bullet}$ in the category of complexes of $A$-modules.

We now compute the cokernel
$\mathcal{E}^{\bullet}_{n+1}$ of $i_n^{\bullet}: C_{n}^{\bullet}\xrightarrow [] {}C_{n+1}^{\bullet}$ for each $n\geq 1$. In addition, we set $\E^{\bu}_1=C^{\bu}_1$. Observe that $\E^l_n=0$ for $l<-n$.
The complex $\mathcal{E}^{\bullet}_{n}$ for each $n\geq 1$ is as follows
$$
0 \rightarrow
\bigoplus_{i\in Q_{0}}I_{i}^{(\mathbf{B}^{-n, 0}_{i})}\stackrel{\widetilde{\partial}^{-n}}{\longrightarrow}
\bigoplus_{i\in Q_{0}}I_{i}^{(\mathbf{B}^{-n+1, 1}_{i})}
\stackrel{\partial^{-n+1}}{\longrightarrow} \bigoplus_{i\in Q_{0}}I_{i}^{(\mathbf{B}^{-n+2, 2}_{i})}\stackrel{}{\longrightarrow}\cdots$$
$$\stackrel{}{\longrightarrow}\bigoplus_{i\in Q_{0}}I_{i}^{(\mathbf{B}^{l, l+n}_{i})}
\stackrel{\partial^{l}}{\longrightarrow} \bigoplus_{i\in Q_{0}}I_{i}^{(\mathbf{B}^{l+1, l+n+1}_{i})}
\stackrel{}{\longrightarrow} \cdots,
$$ where $\widetilde{\partial}^{-n}$ is given by $\widetilde{{\partial}}^{-n}(e_{i}^{\s}\zeta_{(p, q)})=0$ and
\begin{equation*}\widetilde{{\partial}}^{-n}(\alpha^{\s}\zeta_{(p, q)})=\begin{cases}
-\sum\limits_{\beta\in S(\alpha)} e_{s(\alpha)}^{\s}\zeta_{(\beta\widehat{p}, \beta)}, & \text{if $p=\alpha\widehat{p}$ and $\alpha$ is special;}\\
 e_{s(\alpha)}^{\s}\zeta_{(p, q\alpha)},& \text{otherwise},
\end{cases}
\end{equation*} for any $i\in Q_{0}$, $(p, q)\in\mathbf{B}^{-n, 0}_{i}$ and $\alpha\in Q_{1}$ with $t(\alpha)=i$. The above differentials $\p^l$ for $l\geq -n+1$ are
inherited from $C^{\bu}$.

We will describe an explicit decomposition of $\E^{\bu}_n$.
For $i, j\in Q_0$, $l\in \Z$ and $n\geq 0$, set
$$\mathbf{B}^{l, n}_{i, j}=\{(p, q)~| ~(p, q)\in \mathbf{B}^{l}_{i} \text{~with~} q\in Q_n \text{~and~} s(p)=j\}.$$
For each vertex $j\in Q_0$ and $l\geq -1$, denote by
$E^{l}_{j}=\bigoplus\limits_{i\in Q_{0}}I_{i}^{(\mathbf{B}^{l, l+1}_{i, j})}$.
Then we have the following subcomplex $E_{j}^{\bullet}$ of $C^{\bu}_1$
\begin{equation}
\label{injrespro}
0\rightarrow
E^{-1}_{j}\stackrel{\widetilde{\partial}^{-1}}{\longrightarrow}
E^{0}_{j}\stackrel{\p^0}{\longrightarrow}
E^{1}_{j}\stackrel{}{\longrightarrow}
\cdots\stackrel{}{\longrightarrow}
E^{l}_{j}\stackrel{\p^{l}}{\longrightarrow}
E^{l+1}_{j}\stackrel{}{\longrightarrow} \cdots.
\end{equation} Indeed, we have $C^{\bu}_1=\bigoplus_{j\in Q_0}E_{j}^{\bullet}$.

For each vertex $j\in Q_{0}$ and $n\geq 0$,
denote by $Q_{n,j}=\{~p\in Q_{n} | ~t(p)=j \}$.
For a complex $X^{\bullet}=(X^{i}, d^i_{X})_{i\in \Z}$ of $A$-modules, we denote by $X^{\bullet}[1]$ the complex given by $(X^{\bullet}[1])^{i}=X^{i+1}$ and
$d^i_{X[1]}=-d^{i+1}_{X}$ for $i\in\Z$.
For each $n\in \Z$, we denote by $[n]$ the $n$-th power of $[1]$.

\begin{prop} \label{coproductinj} We have an isomorphism $\mathcal{E}^{\bullet}_{n}\cong\bigoplus_{j\in Q_{0}}{(E^{\bullet}_{j}[n-1])}^{(Q_{n-1, j})}$ of complexes of $A$-modules for each $n\geq 1$.
\end{prop}

\begin{proof} For each vertex $i\in Q_{0}$, $l\geq -n$ and path $\g\in Q_{n-1}$,
set $$\Bb(\g)_i^{l, l+n}=\{(p, q)\;|\; (p, q)\in \Bb_i^{l, l+n} \text{~with~} p=\b \g \text{~for some~} \b\in Q_1\}.$$ We observe a subcomplex $E(\g)^{\bu}$
of $\E^{\bu}_n$, whose components are given by
$$E(\g)^l=\bigoplus_{i\in Q_0}I_i^{(\Bb(\g)_i^{l, l+n})}$$ for $l\geq -n$ and
$E(\g)^l=0$ for $l<-n$. For any $(p, q)\in \mathbf{B}^{l, l+n}_{i}$ with $l\geq -n$, we have $l(p)=n$
and then $(p, q)\in \Bb(\widehat{p})_i^{l, l+n}$. Here, $\widehat{p}$ is the truncation of $p$.
Then we have the disjoint union $\Bb_i^{l, l+n}=\bigcup_{\g\in Q_{n-1}}\Bb(\g)_i^{l, l+n}$, and a decomposition
$\E^{\bu}_n=\bigoplus_{\g\in Q_{n-1}}E(\g)^{\bu}$ of complexes.

For each path $\g\in Q_{n-1}$ and $l\geq -n$, we have a bijection
$\Bb(\g)_i^{l, l+n}\xrightarrow[]{}\Bb^{l+n-1, l+n}_{i, t(\g)}$, sending
$(p, q)$ to $(\b, q)$ for $(p, q)\in \Bb(\g)_i^{l, l+n}$
with $p=\b \g$. This bijection induces an isomorphism of complexes
$$f_{\g}^{\bullet}: E(\g)^{\bullet}\xrightarrow []{\sim} E_{t(\g)}^{\bu}[n-1]$$ such that $f_{\g}^l(x\z_{(p, q)})=(-1)^{l(n-1)}x\z_{(\b, q)}$
for $l\geq -n$ and $x\z_{(p, q)}\in I_i\z_{(p, q)}$ with $i\in Q_0$, and $f_{\g}^l=0$ for $l<-n$.
Therefore, we have \begin{equation*}\E^{\bu}_n=\bigoplus_{j\in Q_0}\bigoplus_{\g\in Q_{n-1, j}}E(\g)^{\bu}
\cong \bigoplus_{j\in Q_0}(E_j^{\bu}[n-1])^{(Q_{n-1, j})}.\qedhere\end{equation*}
\end{proof}

We give two examples to describe the cokernel complex $C^{\bu}$ and its subcomplexes $C^{\bu}_n$.

\begin{exm}  Let $Q$ be the following quiver with one vertex and one loop. $$\xymatrix{\scriptstyle{1}\cdot\ar@(ur, dr)[]|{\alpha}}$$ We use the same notation as in
Example \ref{oneandone}. In particular, $A=kQ/J^2$ is a self-injective algebra, and then
$I=DA\cong A$.

Set $\mathbf{B}^{l, n}=\{(p, q)\;| \;(p, q)\in \mathbf{B}^{l} \text{~with~} q\in Q_{n} \}$ for $l\in \Z$ and $n\geq 0$. We have that
\begin{equation*}\Bb^{l, n}=\begin{cases}
\Bb^l, &
\text{if $l=n\geq 0$, or $l<0$ and $n=0$};\\
\emptyset,&\text{otherwise}.
\end{cases}
\end{equation*} The subcomplex $M^{\bullet}$ of the injective Leavitt complex $\I^{\bu}$ of $Q$ is as follows
$$
0\rightarrow
I\z^{0}
\stackrel{\p^{0}}{\longrightarrow}I\z^{1}
\stackrel{}{\longrightarrow}\cdots
\stackrel{}{\longrightarrow} I\z^{l}
\stackrel{\p^{l}}{\longrightarrow} I\z^{l+1}
\stackrel{}{\longrightarrow} \cdots.
$$

We denote by $\Bb^{l, +}=\bigcup_{m>l}\mathbf{B}^{l, m}$ for $l\in \Z$; compare (\ref{integer}). Then $\Bb^{l, +}=\Bb^l$ for $l<0$ and $\Bb^{l, +}=\emptyset$ for $l\geq 0$.
It follows that
the cokernel complex $C^{\bullet}$ is as follows
$$
\cdots
\stackrel{}{\longrightarrow} I\z^{-3}
\stackrel{\p^{-3}}{\longrightarrow} I\z^{-2}
\stackrel{\p^{-2}}{\longrightarrow} I\z^{-1}
\stackrel{}{\longrightarrow} 0.
$$ Observe that $C^{\bu}_1=I[1]$ is a stalk complex concentrated on degree $-1$.
The subcomplex $C^{\bu}_n$ of $C^{\bu}$ for each $n\geq 2$ is as follows
$$
C_n^{\bu}=\quad 0\rightarrow
I\z^{-n}
\stackrel{\p^{-n}}{\longrightarrow}I\z^{-n+1}
\stackrel{}{\longrightarrow}\cdots
\stackrel{}{\longrightarrow} I\z^{-2}
\stackrel{\p^{-2}}{\longrightarrow} I\z^{-1}
\stackrel{}{\longrightarrow} 0.
$$
\end{exm}

\begin{exm}  Let $Q$ be the following quiver with one vertex and two loops. $$\xymatrix{\scriptstyle{1}\cdot \ar@(ul, dl)[]|{\alpha_{1}} \ar@(ur, dr)[]|{\alpha_{2}}}$$
We choose $\alpha_{1}$ to be the special arrow starting at the unique vertex. We use the same notation as in Example \ref{exm}.
Set $\mathbf{B}^{l, n}=\{(p, q)\;| \;(p, q)\in \mathbf{B}^{l} \text{~with~} q\in Q_{n} \}$ for $l\in \Z$ and $n\geq 0$. The subcomplex $M^{\bullet}$ of the injective Leavitt complex $\I^{\bu}$ of $Q$ is as follows
$$
0\rightarrow
I^{(\mathbf{B}^{0, 0})}
\stackrel{\p^{0}}{\longrightarrow}I^{(\mathbf{B}^{1, 1})}
\stackrel{}{\longrightarrow}\cdots
\stackrel{}{\longrightarrow} I^{(\mathbf{B}^{l,l})}
\stackrel{\p^{l}}{\longrightarrow} I^{(\mathbf{B}^{l+1, l+1})}
\stackrel{}{\longrightarrow} \cdots.
$$

We denote by $\mathbf{B}^{l,+}=\bigcup_{m>l}\mathbf{B}^{l, m}$ for each $l\in \Z$; compare (\ref{integer}). We mention that each set $\mathbf{B}^{l,+}$ is infinite.
The cokernel complex $C^{\bullet}$ is as follows
$$
\cdots\stackrel{\partial^{-3}}{\longrightarrow}\mathcal{I}^{-2}\stackrel{\partial^{-2}}{\longrightarrow}
\mathcal{I}^{-1}\stackrel{\widetilde{\partial}^{-1}}{\longrightarrow}
I^{(\mathbf{B}^{0,+})}
\stackrel{\partial^{0}}{\longrightarrow} \cdots
\stackrel{}{\longrightarrow} I^{(\mathbf{B}^{l,+})}
\stackrel{\partial^{l}}{\longrightarrow} I^{(\mathbf{B}^{l+1, +})}
\stackrel{}{\longrightarrow} \cdots.
$$ We write down the subcomplexes $C_1^{\bu}$ and $C_2^{\bu}$ of $C^{\bu}$, explicitly.
\[\resizebox{1.\hsize}{!}{$C_1^{\bu}=\quad 0\stackrel{}{\longrightarrow}
0\stackrel{\quad}{\longrightarrow} I^{(\mathbf{B}^{-1, 0})}
\stackrel{\widetilde{\partial}^{-1}}{\longrightarrow} I^{(\mathbf{B}^{0, 1})}\stackrel{\partial^{0}}{\longrightarrow}
I^{(\mathbf{B}^{1, 2})}\stackrel{}{\longrightarrow}\cdots
\stackrel{}{\longrightarrow} I^{(\mathbf{B}^{l, l+1})}
\stackrel{\partial^{l}}{\longrightarrow} I^{(\mathbf{B}^{l+1, l+2})}
\stackrel{}{\longrightarrow} \cdots.
$}\]
\[\resizebox{1.\hsize}{!}{$C_2^{\bu}=\quad
0\stackrel{}{\longrightarrow} I^{(\mathbf{B}^{-2, 0})}
\stackrel{\partial^{-2}}{\longrightarrow} \begin{matrix}I^{(\mathbf{B}^{-1, 0})}\\\oplus\\ I^{(\mathbf{B}^{-1, 1})}\end{matrix}\stackrel{\widetilde{\partial}^{-1}}{\longrightarrow}
\begin{matrix}I^{(\mathbf{B}^{0, 1})}\\\oplus\\ I^{(\mathbf{B}^{0, 2})}\end{matrix}\stackrel{}{\longrightarrow}\cdots
\stackrel{}{\longrightarrow} \begin{matrix}I^{(\mathbf{B}^{l, l+1})}\\\oplus\\ I^{(\mathbf{B}^{l, l+2})}\end{matrix}
\stackrel{\partial^{l}}{\longrightarrow}\begin{matrix}I^{(\mathbf{B}^{l+1, l+2})}\\\oplus\\ I^{(\mathbf{B}^{l+1, l+3})}\end{matrix}
\stackrel{}{\longrightarrow} \cdots.
$}\] In $C^{\bu}_2$, the differentials $\widetilde{\p}^{-1}$ and $\p^{l}$
for $l\geq 0$ are given by diagonal matrices of morphisms; compare the second paragraph in the proof of Lemma \ref{injres}.
\end{exm}

\subsection{The homotopy categories}
\label{subsection32}

We consider the category $A$-Mod of left $A$-modules. Denote by $\K(A\-\Mod)$ its homotopy category. We will always view a module as a stalk complex concentrated on degree zero.

Recall that for a complex $X^{\bullet}=(X^{i}, d^i_{X})_{i\in \Z}$ of $A$-modules, the complex $X^{\bullet}[1]$ is given by $(X^{\bullet}[1])^{i}=X^{i+1}$ and
$d^i_{X[1]}=-d^{i+1}_{X}$ for $i\in\Z$.
For a chain map $f^{\bullet}: X^{\bullet}\xrightarrow[]{} Y^{\bullet}$ , its \emph{mapping cone} ${\rm Con}(f^{\bullet})$ is
a complex such that ${\rm Con}(f^{\bullet})=X^{\bullet}[1]\oplus Y^{\bullet}$ with the differential
$d^i_{{\rm Con}(f^{\bullet})}=\begin{pmatrix}-d^{i+1}_{X}&0\\f^{i+1}&d^i_{Y}\end{pmatrix}.$
Each triangle in $\K(A\-\Mod)$ is isomorphic to $$\CD
  X^{\bullet} @>f^{\bullet}>> Y^{\bullet}@>{\begin{pmatrix}0\\1\end{pmatrix}} >> {\rm Con}(f^{\bullet}) @>{\begin{pmatrix}1&0\end{pmatrix}}>> X^{\bullet}[1]
\endCD$$ for some chain map $f^{\bullet}$.

Consider the semisimple left $A$-module $kQ_{0}=A/\rad A$.
Recall from subsection \ref{subsection31} that the injective Leavitt complex $\I^{\bu}$ of $Q$ has a subcomplex $M^{\bu}$.

\begin{lem}  \label{injres} The left $A$-module $kQ_{0}$ is quasi-isomorphic to $M^{\bullet}$ as complexes. In other words, $M^{\bu}$ is an injective resolution of the $A$-module
$kQ_0$.
\end{lem}

\begin{proof} We define an $A$-module monomorphism $f^0:kQ_{0}\xrightarrow[]{} M^0$ by
$f^0(e_{i})=e_{i}^{\s}\zeta_{(e_{i}, e_{i})}$ for each $i\in Q_{0}$, and observe that $\p^0\circ f^0=0$. Then we obtain a chain map $f^{\bullet}=(f^l)_{l\in\Z}: kQ_{0}
\xrightarrow []{}M^{\bullet}$ with $f^l=0$ for $l\neq 0$.

Recall that for each $l\geq 0$, $\I^l=M^l\oplus C^l$, $\p^l(M^l)\subseteq M^{l+1}$ and $\p^l(C^l)\subseteq C^{l+1}$. Applying Proposition \ref{propacy}, we obtain $H^l(M^{\bu})=0$ for $l\geq 1$. By Lemma \ref{partialker} $\Ke \p^0\cap M^0$ has a $k$-basis $\{e_i^{\s}\z_{(e_i, e_i)}\;|\;i \in Q_0\}$, and thus $\Ke \p^0\cap M^0=\Im f^0$. Then we are done.
\end{proof}

\begin{lem} \label{projinj} For each vertex $j\in Q_{0}$, the complex $E_{j}^{\bullet}$ in $(\ref{injrespro})$ is quasi-isomorphic to $Ae_{j}[1]$. In other words, $E_{j}^{\bullet}[-1]$ is an injective resolution of the projective $A$-module $Ae_j$.
\end{lem}

\begin{proof} Define an $A$-module monomorphism $\epsilon^{-1}: Ae_{j}\xrightarrow[]{} E^{-1}_{j}$ such that $$\epsilon^{-1}(e_{j})=\sum_{\{\beta\in Q_{1}\;|\; s(\beta)=j\}}\beta^{\s}\zeta_{(\beta, e_{t(\beta)})}.$$ We observe that $\epsilon^{-1}(\alpha)=e_{t(\alpha)}^{\s}\zeta_{(\alpha, e_{t(\alpha)})}$ for $\alpha\in Q_{1}$ with $s(\alpha)=j$. Then we have a chain map by
$\epsilon^{\bullet}=(\epsilon^l)_{l\in\Z}: Ae_{j}[1]
\xrightarrow []{}E^{\bullet}_j$ such that $\epsilon^l=0$ for $l\neq -1$.
Here, we use the fact that $\widetilde{\p}^{-1}\circ\epsilon^{-1}=0$.

For $l\geq 0$, we have $\I^l=E^l_j \oplus L^l$ for $L^l=\bigoplus_{i\in Q_0}I_i^{(\Bb^l_i\setminus \Bb^{l, l+1}_{i, j})}$. Moreover, we observe that $\p^l(E^l_j)\subseteq E^{l+1}_j$ and $\p^{l}(L^l)\subseteq L^{l+1}$. Applying Proposition \ref{propacy},
we obtain $H^l(E^{\bu}_j)=0$ for $l\geq 1$.

We now prove that $H^0(E^{\bu}_j)=0$ and $Ae_j\cong \Ke \widetilde{\p}^{-1}$.
Recall the differential $\widetilde{\p}^{-1}:E^{-1}_j\xra E^0_j$. For each $l\geq -1$, we have the following $k$-basis of $E^l_j$
$$\La^l=\{e_i^{\s}\z_{(p, q)},\aa^{\s}\zeta_{(p, q)}\;|\;i\in Q_0, (p, q)\in \Bb^{l, l+1}_{i,j} \text{~and~} \aa\in Q_1\text{~with~} t(\aa)=i\}.$$
Consider a subset $\La^l_0=\{e_i^{\s}\zeta_{(p, q)}\;|\;i\in Q_0, (p, q)\in \Bb^{l, l+1}_{i,j}\}$.
Set
$$\La^{-1}_2=\{\aa^{\s}\z_{(\aa, e_{t(\aa)})}\;|\;\aa\in Q_1 \text{~is special with~} s(\aa)=j\}.$$
Take $\La^{-1}_1=\La^{-1}\setminus (\La^{-1}_0\cup\La^{-1}_2)$.
Then we have a disjoint union $\La^{-1}=\La^{-1}_0\cup\La^{-1}_1\cup\La^{-1}_2$.
The triple $(\widetilde{\p}^{-1}, \La^{-1},\La^0)$ satisfies Condition (Z) in
Lemma \ref{kerim2}. Then we infer that
$\La^{-1}_0\cup\{\sum_{\{\beta\in Q_{1}\;|\;s(\beta)=j\}}\beta^{\s}\zeta_{(\beta, e_{t(\beta)})}\}$ is a $k$-basis of $\Ke \widetilde{\p}^{-1}$ and
that $\widetilde{\p}^{-1}(\La^{-1}_1)$ is a $k$-basis of $\Im \widetilde{\p}^{-1}$.
By Lemma \ref{partial}(2), we obtain $\widetilde{\p}^{-1}(\La^{-1}_1)=\La^{0}_0$.
By Lemma \ref{partialker}, the set $\La^0_0$ is a $k$-basis of $\Ke \p^0\cap E_j^0$.
It follows that $\Im \widetilde{\p}^{-1}=\Ke \p^0\cap E_j^0$, that is, $H^0(E^{\bu}_j)=0$. We observe that $\Im \epsilon^{-1}=\Ke \widetilde{\p}^{-1}$.
This completes the proof that $\epsilon^{\bullet}$ is a quasi-isomorphism.
\end{proof}

Recall that $\E^{\bu}_1=C^{\bu}_1$. By the decompositions $C^{\bu}_1=\bigoplus_{j\in Q_0}E^{\bu}_j$ and $A=\bigoplus_{j\in Q_0} Ae_j$, we have the following consequence.

\begin{cor} \label{cq} The complex $\E_1^{\bullet}$ is quasi-isomorphic to $A[1]$.
In other words, the complex $\E_1^{\bullet}[-1]$ is an injective resolution of the regular module $A$.\hfill $\square$
\end{cor}

We now recall some terminology and facts on triangulated categories.
For a triangulated category $\mathcal{T}$, a \emph{thick} subcategory of $\mathcal{T}$
is a triangulated subcategory of $\mathcal{T}$ which is closed under direct summands. Let $\mathcal{S}$ be a class of objects in $\mathcal{T}$.
We denote by thick$\langle\mathcal{S}\rangle$ the smallest thick subcategory of $\mathcal{T}$
containing $\mathcal{S}$.
If $\mathcal{T}$ has arbitrary coproducts, we denote by
${\rm Loc}\langle\mathcal{S}\rangle$ the smallest triangulated subcategory of
$\mathcal{T}$ which contains $\mathcal{S}$ and is closed under arbitrary coproducts.
By \cite[Proposition 3.2]{bn} we have that thick$\langle\mathcal{S}\rangle\subseteq{\rm Loc}\langle \mathcal{S}\rangle$.

The following standard fact will be used.

\begin{lem}\label{cc}Let $\mathcal{T}$ be a triangulated category with arbitrary coproducts. Let $S$ be an object in $\mathcal{T}$ and $\mathcal{T}'$ be a triangulated subcategory of $\mathcal{T}$ satisfying $\Hom_{\mathcal{T}}(S, T')=0$ for each
$T'\in \mathcal{T}'$. Then any object $X\in {\rm Loc}\langle \mathcal{S}\rangle$
satisfies $\Hom_{\mathcal{T}}(X, T')=0$ for each
$T'\in \mathcal{T}'$.
\end{lem}

\begin{proof} It suffices to observe that $\{Y\in \mathcal{T}\;|\; \Hom_{\mathcal{T}}(Y, T')=0 \text{~for each~} T'\in\mathcal{T}'\}$ is a triangulated subcategory of $\mathcal{T}$, which is closed under arbitrary coproducts.
\end{proof}

For a triangulated category $\mathcal{T}$ with arbitrary coproducts, an object $M$ in $\mathcal{T}$ is \emph{compact}
if the functor ${\rm Hom}_{\mathcal{T}}(M,-)$ commutes with arbitrary coproducts.
Denote by $\mathcal{T}^{c}$ the full subcategory consisting of compact objects; it is a thick subcategory.

A triangulated category $\mathcal{T}$ with arbitrary coproducts is \emph{compactly generated} \cite{ke, n1} if there exists a set $\mathcal{S}$ of compact objects such that
any nonzero object $T$ satisfies that ${\rm Hom}_{\mathcal{T}}(S,T[n])\neq 0$
for some $S\in\mathcal{S}$ and $n\in\Z$.  This  is equivalent to the condition that
$\mathcal{T}={\rm Loc}\langle\mathcal{S}\rangle$, in which case we have $\mathcal{T}^{c}$=thick$\langle\mathcal{S}\rangle$; see \cite[Lemma 3.2]{n1}. If the above set $\mathcal{S}$
consists of a single object $S$, we call $S$ a \emph{compact generator} of $\mathcal{T}$.

\begin{lem} \label{comg}Suppose that $\mathcal{T}$ is a compactly generated triangulated category
with a compact generator $X$. Let $\mathcal{T}'\subseteq \mathcal{T}$ be a triangulated
subcategory closed under arbitrary coproducts. Assume that there exists a triangle
\begin{align*}\CD
 X @>>>Y@>>>Z @>>>X[1]
\endCD
\end{align*} such that $Y\in \mathcal{T}'$ and $Z$ satisfies $\Hom_{\mathcal{T}}(Z, T')=0$
for each $T'\in\mathcal{T}'$. Then $Y$ is a compact generator of $\mathcal{T}'$.
\end{lem}

\begin{proof} For each $T'\in\mathcal{T}'$, we apply the cohomological functor $\Hom_{\mathcal{T}}(-, T')$
to the above triangle. Then we obtain a functorial isomorphism $\Hom_{\mathcal{T}}(X, T')\cong \Hom_{\mathcal{T}'}(Y, T')$. The result follows immediately, since $X$ is a compact generator of $\mathcal{T}$.
\end{proof}

Let $A\-\Inj$ be the category of injective $A$-modules. Denote by $\K(A\-\Inj)$ the homotopy category
of complexes of injective $A$-modules, which is a triangulated subcategory of $\K(A\-\Mod)$ that is closed under coproducts. By \cite[Proposition 2.3(1)]{kr} $\K(A\-\Inj)$ is a compactly generated triangulated category.

Denote by $A$-mod the category of finitely generated $A$-modules and by $\D^b(A\-\mod)$ its bounded derived category.
Recall that each bounded complex of $A$-modules admits a quasi-isomorphism to a bounded-below complex of injective $A$-modules.
This gives rise to a full embedding $\D^b(A\-\mod)\hookrightarrow\K(A\-\Inj)$, which induces a triangle equivalence
$\D^b(A\-\mod)\xrightarrow[]{\sim}\K(A\-\Inj)^c$; see \cite[Proposition 2.3(2)]{kr}.

We point out that the following lemma is contained in the proof
of \cite[Theorem 2.5]{cy}.

\begin{lem}\label{cominj} The complex $M^{\bu}$ is a compact generator of $\K(A\-\Inj)$.
\end{lem}

\begin{proof} Observe that thick$\langle kQ_0\rangle=\D^b(A\-\mod)$. Recall Lemma \ref{injres} that $M^{\bu}$ is an injective resolution of the $A$-module $kQ_0$. It follows from \cite[Proposition 2.3]{kr} that
$M^{\bu}$ is a compact
object in $\K(A\-\Inj)$ and ${\rm Loc}\langle M^{\bu}\rangle=\K(A\-\Inj)$.
\end{proof}

\begin{comment}
Let $0\xrightarrow []{}X_{0}\xrightarrow []{}X_{1}\xrightarrow []{}X_{2}\xrightarrow[]{}\cdots$ be a sequence of objects and morphisms in a triangulated
category with arbitrary coproducts.
Recall that the \emph{homotopy colimit} \cite{bn, n1, n2} of the sequence, denoted by $hocolim(X_{i})$, is given by the triangle

\begin{align} \label{eq:m}\CD
  \bigoplus^{\infty}_{i=0}X_{i} @>1-\shift>> \bigoplus^{\infty}_{i=0}X_{i}@>>> hocolim(X_{i}) @>>> \bigoplus^{\infty}_{i=0}X_{i}[1]
\endCD
\end{align}
where the shift map $\bigoplus\limits^{\infty}_{i=0}X_{i}\xrightarrow []{\shift}\bigoplus\limits^{\infty}_{i=0}X_{i}$ is the direct sum of $X_{i}\xrightarrow[] {} X_{i+1}$ and $[1]$ is the translation functor of the triangulated category. The triangle $(\ref{eq:m})$ is Milnor's triangle \cite{m}.
\end{comment}

Recall from subsection 3.1 the subcomplexes $C^{\bu}_n$ of the cokernel complex $C^{\bu}$ for $n\geq 1$
and $C^{\bullet}=\lim\limits_{\xrightarrow[] {}} C_{n}^{\bullet}$. Then we have an exact sequence of complexes \begin{align*}\CD
 0@>>> \bigoplus^{\infty}\limits_{n=1}C_{n}^{\bullet}@>{1-\shift}>>\bigoplus^{\infty}\limits_{n=1}C_{n}^{\bullet} @>>> C^{\bullet}@>>>0,
\endCD
\end{align*} where the restriction of the chain map $1\-\shift$  on $C^{\bu}_n$ equals $\begin{pmatrix}1\\-i_{n}^{\bu}\end{pmatrix}:C^{\bu}_n\xrightarrow []{}C^{\bu}_n\oplus C_{n+1}^{\bu}\subseteq \bigoplus^{\infty}\limits_{n=1}C_{n}^{\bullet}$ .
This sequence is in fact split exact in each component, since each component of the complex $\bigoplus^{\infty}\limits_{n=1}C_{n}^{\bullet}$ is an injective module. So it gives rise to a triangle
\begin{align}\label{eq:v}\CD
  \bigoplus^{\infty}\limits_{n=1}C_{n}^{\bullet} @>1-\shift>>\bigoplus^{\infty}\limits_{n=1}C_{n}^{\bullet}@>>> C^{\bullet} @>>> (\bigoplus^{\infty}\limits_{n=1}C_{n}^{\bullet})[1]
\endCD
\end{align} in $\mathbf{K}(A\-\Inj)$.
In other words, $C^{\bullet}$ is the homotopy colimit of $C^{\bu}_n$;
see \cite[Definition 2.1]{bn}.

\begin{prop}\label{clocalizing}  The cokernel complex $C^{\bullet}$ belongs to ${\rm Loc}\langle \mathcal{E}_1^{\bullet}\rangle\subseteq \mathbf{K}(A\-\Inj)$.
\end{prop}

\begin{proof} Recall $\E^{\bu}_1=C^{\bu}_1=\bigoplus_{j\in Q_0}E^{\bu}_j$. By Proposition \ref{coproductinj} we have that $\mathcal{E}_{n}^{\bullet}\in {\rm Loc}\langle \mathcal{E}^{\bullet}_1\rangle$ for $n\geq 1$.
The exact sequence \begin{align*}\CD
 0@>>> C_{n}^{\bullet}@>{i^{\bu}_n}>>C_{n+1}^{\bullet}@>>> \E_{n+1}^{\bullet}@>>>0
\endCD
\end{align*} induces a triangle
\begin{align*}\CD
  C_{n}^{\bullet} @>i^{\bu}_n>>C_{n+1}^{\bullet}@>>>\E_{n+1}^{\bullet} @>>>C_{n}^{\bullet}[1]
\endCD
\end{align*} in the category $\mathbf{K}(A\-\Inj)$.
By induction,
we have that $C_{n}^{\bullet}\in {\rm Loc}\langle \mathcal{E}_1^{\bullet}\rangle$ for each $n\geq 1$. The triangle $(\ref{eq:v})$ implies
$C^{\bullet}\in {\rm Loc}\langle \mathcal{E}_1^{\bullet}\rangle$. Then we are done.
\end{proof}

\subsection{An explicit compact generator}

Denote by ${\mathbf{K}_{\rm ac}}(A\-\Inj)$ the full subcategory of $\mathbf{K}(A\-\Inj)$
formed by acyclic complexes of injective $A$-modules. This category is
a compactly generated triangulated category such that its subcategory consisting of compact objects is triangle equivalent to the singularity category
of $A$; see \cite[Corollary 5.4]{kr}. The category ${\mathbf{K}_{\rm ac}}(A\-\Inj)$
is called the stable derived category of $A$; see \cite[Definition 5.1]{kr}.

Here, we recall that the singularity category $\D_{\rm sg}(A)$ of $A$ is
the quotient triangulated category of $\D^{b}(A\-\mod)$ by
the full subcategory formed by perfect complexes; see \cite{bu,o}.
A complex in $\D^{b}(A\-\mod)$ is perfect provided that it is
isomorphic to a bounded complex consisting of finitely generated projective modules.
We mention that the singularity category $\D_{\rm sg}(A)$ is described by \cite[Theorem 7.2]{s} and \cite[Theorem 3.8]{c1}, in the case that $A$ is with
radical square zero.

The following theorem is the main result of this section.

\begin{thm}\label{tc} Let $Q$ be a finite quiver without sinks, and let $A=kQ/J^2$. Then the injective Leavitt complex $\mathcal{I}^{\bullet}$ of $Q$
is a compact generator of the category ${\rm \mathbf{K}_{ac}}(A\-\Inj)$.
\end{thm}

\begin{proof} By Proposition \ref{propacy}, we have $\mathcal{I}^{\bullet}\in{\rm \mathbf{K}_{ac}}(A\-\Inj)$.
Recall the cokernel complex
$C^{\bullet}={\rm Coker}(\iota^{\bullet})$, where $\iota^{\bullet}: M^{\bullet}\xrightarrow[]{}
\mathcal{I}^{\bullet}$ is the inclusion chain map. Then we have the following exact sequence of complexes
\begin{align*}\CD
 0@>>> M^{\bullet}@>{\iota^{\bu}}>>\I^{\bullet} @>>> C^{\bullet}@>>>0.
\endCD
\end{align*} The above sequence splits in each component, since $M^{\bu}$ consists of injective modules.
This gives rise to a triangle in $\mathbf{K}(A\-\Inj)$\begin{align}\label{eq:q}\CD
  M^{\bullet}@>\iota^{\bullet}>> \mathcal{I}^{\bullet}@>>>C^{\bullet} @>>> M^{\bullet}[1].
\endCD
\end{align}

For any complex $X^{\bu}\in\K_{\rm ac}(A\-\Inj)$, we have ${\rm Hom}_{\mathbf{K}(A\-\Mod)}(A, X^{\bullet})=0$. Here, we recall the canonical isomorphism ${\rm Hom}_{\mathbf{K}(A\-\Mod)}(A, Y^{\bullet}[n])\cong H^{n}(Y^{\bu})$ for each complex $Y^{\bu}$ and $n\in\Z$.
Recall from Corollary \ref{cq} that the complex $\E^{\bu}_1[-1]$ is an injective resolution of $A$.
We apply \cite[Lemma 2.1]{kr} to deduce the following isomorphism
$${\rm Hom}_{\mathbf{K}(A\-\Inj)}(\mathcal{E}_1^{\bullet}, X^{\bullet})\cong{\rm Hom}_{\mathbf{K}(A\-\Mod)}(A[1], X^{\bullet})=0.$$

Recall from Proposition \ref{clocalizing} that $C^{\bullet}\in {\rm Loc}\langle \mathcal{E}_1^{\bullet}\rangle$. Then by Lemma \ref{cc} we have ${\rm Hom}_{\mathbf{K}(A\-\Inj)}(C^{\bullet}, X^{\bullet})=0$ for any $X^{\bullet}\in \mathbf{K}_{\rm ac}(A\-\Inj)$. Recall from Lemma \ref{cominj}
that $M^{\bu}$ is a compact generator of $\K(A\-\Inj)$.
By the triangle ($\ref{eq:q}$) and Lemma \ref{comg}, we are done.
\end{proof}

\section{The injective Leavitt complex as a differential graded bimodule}
\label{sfour}

In this section, we endow the injective Leavitt complex with a differential graded module structure over the corresponding Leavitt path algebra.

\subsection{The Leavitt path algebra and module structure}
\label{subsection41}

Let $k$ be a field and $Q$ be a finite quiver without sinks.
We will endow the injective Leavitt complex of $Q$ with a Leavitt path algebra module structure.
Recall from \cite{ap, amp}
the notion of the Leavitt path algebra.

\begin{defi} \label{defleavitt}
The \emph{Leavitt path algebra} $L_{k}(Q)$ of $Q$
is the $k$-algebra generated by the set $\{e_{i}\;|\;i\in Q_{0}\}\cup \{\alpha\;|\;\alpha\in Q_{1}\}
\cup\{\alpha^{*}\;|\;\alpha\in Q_{1}\}$ subject to the following relations:

(0) $e_{i}e_{j}=\delta_{i, j}e_{i}$ for every $i, j\in Q_{0}$;

(1) $e_{t(\alpha)}\alpha=\alpha e_{s(\alpha)}=\alpha$ for all $\alpha\in Q_{1}$;

(2) $e_{s(\alpha)}\alpha^{*}=\alpha^{*} e_{t(\alpha)}=\alpha^{*}$ for all $\alpha\in Q_{1}$;

(3) $\alpha\beta^{*}=\delta_{\alpha, \beta}e_{t(\alpha)}$ for all $\alpha,\beta\in Q_{1}$;

(4) $\sum_{\{\alpha\in Q_{1}\;|\;s(\alpha)=i\}}\alpha^{*}\alpha=e_{i}$ for every $i\in Q_{0}$.\hfill $\square$
\end{defi}

Here, $\delta$ is the Kronecker symbol. The relations $(3)$ and $(4)$ are called
\emph{Cuntz-Krieger relations}. The elements $\alpha^{*}$ for $\alpha\in Q_{1}$ are called \emph{ghost arrows}.

There is an alternative description of $L_k(Q)$.
Let $\overline{Q}$ be the \emph{double quiver} obtained from $Q$ by adding for each arrow $\alpha$ in $Q$
an arrow $\alpha^{*}$ in the opposite direction.
Then the Leavitt path algebra $L_{k}(Q)$ is isomorphic to the quotient algebra of the path algebra $k\overline{Q}$ of $\overline{Q}$
modulo the ideal generated by $\{\alpha\beta^{*}-\delta_{\alpha, \beta}e_{t(\alpha)},
\sum_{\{\g\in Q_{1}\;|\; s(\g)=i\}}\g^{*}\g-e_{i}\;|\;\alpha,\beta\in Q_{1}, i\in Q_{0}\}.$

If $p=\alpha_{n}\cdots\aa_2\alpha_{1}$ is a path in $Q$ of length $n\geq 1$, we define $p^{*}=\alpha_{1}^*\aa_2^*\cdots\alpha_{n}^*$.
For convention, we set $e_{i}^*=e_{i}$ for $i\in Q_{0}$.
We observe by $(2)$ that for paths $p, q$ in $Q$, $p^*q=0$ for $t(p)\neq t(q)$.
Consider the relation $(3)$. We have the following fact; see \cite[Lemma 3.1]{t}.

\begin{lem}\label{mul}Let $p$, $q$, $\g$ and
$\eta$ be paths in $Q$ with $t(p)=t(q)$ and $t(\g)=t(\eta)$.
Then in $L_k(Q)$ we have
\begin{equation*}
(p^*q)(\g^*\eta)=
\begin{cases}
(\g' p)^*\eta, & \text{if $\g=\g'q$};\\
p^*\eta, & \text{if $q=\g$};\\
p^*(q'\eta), & \text{if $q=q'\g$};\\
0, &\text{otherwise}.
\end{cases}
\end{equation*} Here, $\g'$ and $q'$ are some nontrivial paths in $Q$.\hfill $\square$
\end{lem}

By the above lemma, we deduce that
the Leavitt path algebra $L_k(Q)$ is spanned by the following set
$$\{p^*q\;|\; p, q \text{~are paths in~} Q \text{~with~} t(p)=t(q)\};$$
see \cite[Lemma 1.5]{ap}, \cite[Corollary 3.2]{t} or \cite[Corollary 2.2]{c2}. By $(4)$, this set is not $k$-linearly independent in general.

The following result is \cite[Theorem 1]{aajz}.

\begin{lem} \label{lbasis} The following elements form a $k$-basis of the Leavitt path algebra $L_k(Q)$:\begin{enumerate}\item[(1)] $e_i$, $i\in Q_0$;

\item[(2)]  $p, p^*$, where $p$ is a nontrivial path in $Q$;

\item[(3)]  $p^*q$ with $t(p)=t(q)$, where
$p=\alpha_{m}\cdots\alpha_{1}$ and $q=\beta_{n}\cdots\beta_{1}$ are nontrivial paths of $Q$ such that $\alpha_{m}\neq \beta_{n}$, or $\alpha_{m}=\beta_{n}$ which is not special.\hfill $\square$\end{enumerate}
\end{lem}

We denote by $\Lambda$ the $k$-basis of $L_k(Q)$ given in Lemma \ref{lbasis}.
Define a map $\chi: \La\xra\bigcup_{l\in\Z, i\in Q_{0}}\Bb^l_i$ by $\chi(e_i)=(e_i, e_i)$, $\chi(p)=(e_{t(p)},p)$, $\chi(p^*)=(p, e_{t(p)})$ and $\chi(p^*q)=(p, q)$.
The map $\chi$ is bijective. Then we identify $\La$ with the set of admissible pairs in $Q$; see Definition \ref{da}.
A nonzero element $x$ in $L_k(Q)$ can be written uniquely in the following form
\begin{equation*}
x=\sum_{i=1}^{m}\lambda_{i}p_{i}^{\ast}q_{i}
\end{equation*}
with $\lambda_{i}\in k$ nonzero scalars
and $(p_{i}, q_{i})$ pairwise distinct admissible pairs in $Q$.

From now on, $Q$ is a finite quiver without sinks and $B=L_k(Q)^{\rm op}$ is the opposite algebra of $L_k(Q)$. The multiplication $``\c"$ in $B$ is given by $a\c b=ba$.
%Recall from Definition \ref{definj} that the injective Leavitt complex
%$\mathcal{I}^{\bullet}=(\I^{l}, \partial^{l})_{l\in\Z}$
%and
%$\mathcal{I}^{l}=\bigoplus_{i\in Q_{0}}{I_{i}}^{(\mathbf{B}^{l}_{i})}$.

We define a right $B$-module structure on $\I^{\bullet}$.
For each vertex $j\in Q_{0}$ and each arrow $\alpha\in Q_{1}$, we define the right action $``\cdot"$ on
$\mathcal{I}^{l}$ for any $l\in\Z$ as follows.
For any element $x\zeta_{(p, q)}\in I_{i}\zeta_{(p, q)}$ with $i\in Q_{0}$ and
$(p, q)\in \mathbf{B}^{l}_{i}$, we set

\begin{equation}\label{action1} x\zeta_{(p, q)}\cdot e_{j}=\delta_{j, s(p)}x\zeta_{(p, q)};\end{equation}

\begin{equation}\label{action2}x\zeta_{(p, q)}\cdot \alpha=\begin{cases}
 \delta_{\alpha, \alpha_{1}}x\zeta_{(\widetilde{p}, q)}, & \text{if $p=\widetilde{p}\alpha_{1}$};\\
 \delta_{s(\alpha), t(q)}x\zeta_{(e_{t(\alpha)}, \alpha q)}, & \text{if $l(p)=0$};
\end{cases}
\end{equation}

\vskip 5pt

\begin{equation}\label{action3}x\zeta_{(p, q)}\cdot \alpha^{*}=\begin{cases}
x\zeta_{(e_{s(\alpha)}, \widehat{q})}-\sum\limits_{\beta\in S(\alpha)}
x\zeta_{(\beta, \beta\widehat{q})}, & \begin{matrix}\text{if~}l(p)=0, ~~q=\alpha\widehat{q}\\ \text{and~~} \alpha \text{~~is special;}\end{matrix}\\
\delta_{s(p),t(\alpha)}x\zeta_{(p\alpha,  q)}, & \text{otherwise}.
\end{cases}
\end{equation}
Here for the notation, a path $p=\aa_n\cdots \aa_{2}\aa_1$ of length $n\geq 2$ has two truncations
$\widetilde{p}=\aa_n\cdots \aa_{2}$ and $\widehat{p}=\aa_{n-1}\cdots \aa_1$. If $p=\aa$ is an arrow, $\widetilde{p}=e_{t(\aa)}$ and $\widehat{p}=e_{s(\aa)}$.
We recall that $S(\alpha)=\{\beta\in Q_{1}\;| \; s(\beta)=s(\alpha), \beta\neq\alpha\}$ for a special arrow $\aa$.

We have the following observations.
\begin{equation}
\begin{cases}
\begin{split}
\label{remarkaction}x\zeta_{(p, q)}\cdot \alpha=0, & \text{~~if $s(\alpha)\neq s(p)$};\\
x\zeta_{(p, q)}\cdot \alpha^{*}=0, &\text{~~if $s(p)\neq t(\alpha)$}.
\end{split}
\end{cases}
\end{equation}

%We may write the action $%(\ref{action1}), (\ref{action2})$ and $(\ref{action3})$ as a left $L_k(Q)$ action on $\I^{\bu}$ and view it as right $B$-action. 

We have a left $L_k(Q)$-action on $\I^{\bu}$ by the right $B$-action. But we write $(\ref{action1}), (\ref{action2})$ and $(\ref{action3})$ as right $B$-action to avoid confusion of two left actions, since the injective Leavitt complex $\I^{\bu}$ has already a left $A$-action.

\begin{comment}
The following fact is immediate from the
definition of the right $B$-action.

\begin{lem} \label{lmodulep}
$(1)$ For $p$ and $q$ paths of $Q$ with $l(p)>0$, we have \begin{equation*}x\zeta_{(p, e_{t(p)})}\cdot q=\begin{cases}
x\zeta_{(e_{t(q')}, q')}, & \text{if $q=q'p$};\\
x\zeta_{(p',  e_{t(p')})}, & \text{if $p=p'q$};\\
x\zeta_{(e_{t(p)},  e_{t(p)})}, & \text{if $p=q$};\\
0, &\text{otherwise}.
\end{cases}
\end{equation*} Here, $p'$ and $q'$ are nontrivial paths in $Q$.

$(2)$ For an admissible pair $(p, q)$ of $Q$, we have $x\zeta_{(e_{t(q)}, q)}\cdot p^{*}=x\zeta_{(p, q)}$.
\end{lem}
\end{comment}

\begin{lem} \label{lmodule} The above action makes the injective Leavitt complex $\mathcal{I}^{\bullet}$ of $Q$
a right $B$-module.
\end{lem}
\begin{proof} It suffices to prove that the above right action on $\I^{\bu}$
satisfies the opposite defining relations of the Leavitt path algebra $L_k(Q)$.

In what follows, we fix $x\z_{(p, q)}\in I_i\z_{(p, q)}\subseteq \I^{l}$.
For $(0)$, we observe that $x\zeta_{(p, q)}\cdot (e_{j}\circ e_{j'})=\delta_{j, j'}
x\zeta_{(p, q)}\cdot e_{j}.$

For $(1)$, we have that\[
\begin{aligned}
x\zeta_{(p, q)}\cdot (e_{s(\alpha)}\circ \alpha)&=(x\zeta_{(p, q)}\cdot e_{s(\alpha)})\cdot\alpha\\
&=\delta_{s(\alpha), s(p)}x\zeta_{(p, q)}\cdot \alpha\\
&=x\zeta_{(p, q)}\cdot \alpha,
\end{aligned}
\] where the last equality uses $(\ref{remarkaction})$.

We have
\[
\begin{aligned}
x\zeta_{(p, q)}\cdot (\alpha\circ e_{t(\alpha)})&=(x\zeta_{(p, q)}\cdot \alpha)\cdot e_{t(\alpha)}\\
&=\begin{cases}
 \delta_{\alpha, \alpha_{1}}\delta_{t(\alpha), t(\alpha_{1})}x\zeta_{(\widetilde{p}, q)}, & \text{if $p=\widetilde{p}\alpha_{1}$};\\
 \delta_{s(\alpha), t(q)}x\zeta_{(e_{t(\alpha)}, \alpha q)}, & \text{if $l(p)=0$}.
\end{cases}\\
&=x\zeta_{(p, q)}\cdot \alpha,
\end{aligned}
\] which proves $(1)$. Similarly, we prove $(2)$.

For $(3)$, we have that
\[
\begin{aligned}
&x\zeta_{(p, q)}\cdot (\beta^{*}\circ \alpha)
=(x\zeta_{(p, q)}\cdot\beta^{*})\cdot \alpha\\
&=\left\{\begin{array}{ll}
\delta_{s(\alpha), s(\beta)}x\zeta_{(e_{t(\alpha)}, \alpha\widehat{q})}-\sum\limits_{\gamma\in S(\beta)}
\delta_{\alpha, \gamma}x\zeta_{(e_{t(\gamma)},\gamma\widehat{q})}, &
\begin{matrix}\text{if~~} l(p)=0, ~~q=\beta\widehat{q}\\ \text{and~~}\beta \text{~~is~~special;} \end{matrix}\\
 \delta_{\alpha,\beta}\delta_{s(p),t(\b)}x\zeta_{(p, q)}, & \text{otherwise}.
 \end{array}\right.\\
&=\delta_{\alpha, \beta}\delta_{s(p),t(\beta)}x\zeta_{(p, q)}\\
&=x\zeta_{(p, q)}\cdot \delta_{\alpha, \beta}e_{t(\alpha)}.
\end{aligned}
\] Here for the third equality, we use the following fact in the case that
$l(p)=0$, $q=\beta\widehat{q}$ and $\beta$ is special: if $\alpha=\beta$,
we have that $s(\alpha)=s(\beta), s(p)=t(\b)$ and $\gamma\neq\alpha$ for all $\gamma\in S(\alpha)$; if $\aa\neq \b$ and $s(\aa)=s(\b)$, we have $\aa\in S(\b)$.

For $(4)$, we fix a vertex $j\in Q_{0}$. If $\alpha\in Q_{1}$ with $s(\alpha)=j$ is special, then
\begin{equation*}(x\zeta_{(p, q)}\cdot\alpha)\cdot\alpha^{*}=\begin{cases}
\delta_{\alpha, \alpha_{1}}x\zeta_{(p, q)}, &\text{if $p=\widetilde{p}\alpha_{1}$};\\
\delta_{j, t(q)}(x\zeta_{(e_{t(q)},  q)}-\sum\limits_{\beta\in S(\alpha)}x\zeta_{(\beta, \beta q)}), & \text{if $l(p)=0$}.\\
\end{cases}
\end{equation*} If $\alpha\in Q_{1}$ with $s(\alpha)=j$ is not special, we have
\begin{equation*}(x\zeta_{(p, q)}\cdot\alpha)\cdot\alpha^{*}=\begin{cases}
\delta_{\alpha, \alpha_{1}}x\zeta_{(p, q)}, & \text{if $p=\widetilde{p}\alpha_{1}$};\\
\delta_{j, t(q)}x\zeta_{(\alpha,  \alpha q)}, & \text{if $l(p)=0$}.
\end{cases}
\end{equation*} We put the two cases together to obtain the following identity.
\begin{equation*}
\begin{split}
x\zeta_{(p, q)}\cdot(\sum_{\{\alpha\in Q_{1}\;|\;s(\alpha)=j\}}\alpha\circ \alpha^{*})
&=\sum_{\{\alpha\in Q_{1}\;|\; s(\alpha)=j\}}(x\zeta_{(p, q)}\cdot\alpha)\cdot \alpha^{*}\\
&=\begin{cases}
\delta_{j, s(p)}x\zeta_{(p, q)}, & \text
{if $p=\widetilde{p}\alpha_{1}$};\\
 \delta_{j, t(q)}x\zeta_{(p, q)}, & \text{if $l(p)=0$}.
 \end{cases}\\
&=\delta_{j, s(p)}x\zeta_{(p, q)}\\
&=x\zeta_{(p, q)}\cdot e_{j}.\qedhere
\end{split}
\end{equation*}
\end{proof}

The following observation will be used in the next subsection. It gives an intuitive description of the $B$-module action on $\I^{\bu}$.

\begin{lem} \label{fact}Let $(p, q)$ be an admissible pair in $Q$.
\begin{enumerate}
\item[(1)] We have $\sum_{i\in Q_0}e_i^{\s}\z_{(e_i, e_i)}\cdot p^*q=e_{s(q)}^{\s}\z_{(p, q)}.$

\item[(2)] For each arrow $\b\in Q_1$ with $t(\b)=i$, we have $\b^{\s}\z_{(e_i, e_i)}\cdot p^*q=\delta_{i, s(q)}\b^{\s}\z_{(p, q)}$.
\end{enumerate}
\end{lem}

\begin{proof} Since $(p, q)$ is an admissible pair in $Q$, we are in the second subcases in (\ref{action2}) and (\ref{action3}) for the right action of $p^*q$.
Then the identities follow from direct calculation.
\end{proof}

\subsection{The differential graded module structure}
\label{subsection42}

We recall from \cite{ke} some notation on differential graded modules.
Let $A=\bigoplus_{n\in\Z}A^{n}$ be a $\Z$-graded algebra. For a (left) graded $A$-module $M=\bigoplus_{n\in\Z}M^n$,
elements $m$ in $M^{n}$ are said to be homogeneous of degree $n$, denoted by $|m|=n$.

A \emph{differential graded algebra} (dg algebra for short) is a $\Z$-graded algebra $A$ with a differential
$d:A \xra A$ of degree one such that $d(ab)=d(a)b+(-1)^{|a|}ad(b)$ for homogenous elements
$a,b\in A$.

A \emph{(left) differential graded} $A$-module (dg $A$-module for short) $M$
is a graded $A$-module $M=\bigoplus_{n\in \mathbb{Z}}M^{n}$ with a differential $d_{M}:M\xra M$
of degree one such that $d_{M}(a{\cdot} m)=d(a){\cdot} m+(-1)^{|a|}a{\cdot} d_{M}(m)$ for homogenous
elements $a\in A$ and $m\in M$. A morphism of dg $A$-modules is a morphism of $A$-modules
preserving degrees and commuting with differentials.
A \emph{right differential graded} $A$-module (right dg $A$-module for short) $N$
is a right graded $A$-module $N=\bigoplus_{n\in \Z}N^{n}$ with a differential $d_{N}:N\xra N$
of degree one such that $d_{N}(m\cdot a)=d_{N}(m)\cdot a+(-1)^{|m|}m\cdot d(a)$ for homogenous
elements $a\in A$ and $m\in N$. Here, we use central dots to denote the $A$-module action.

Let $B$ be another dg algebra.
Recall that a \emph{dg $A$-$B$-bimodule} $M$ is a left dg $A$-module as well as a right dg
$B$-module such that $(a\cdot m)\cdot b=a\cdot (m\cdot b)$ for $a\in A$, $m\in M$ and $b\in B$.

Recall that $Q$ is a finite quiver without sinks and that $L_k(Q)$ is the Leavitt path algebra of $Q$. Set $|e_{i}|=0$, $|\alpha|=1$ and $|\alpha^{*}|=-1$ for $i\in Q_{0}$ and $\alpha\in Q_{1}$.
Then we obtain a $\Z$-grading on $L_k(Q)$ and write $L_k(Q)=\bigoplus_{n\in\Z} L_k(Q)^{n}$. In what follows, we write $B=L_k(Q)^{
\rm op}$, which is graded by $B^n=L_k(Q)^n$. We view $B$ as a dg algebra with trivial differential.

Consider $A=kQ/J^{2}$ as a dg algebra concentrated on degree zero. Recall the injective Leavitt complex $\I^{\bu}=\bigoplus_{l\in\Z}\I^l$, which is a left dg $A$-module.
By Lemma \ref{lmodule}, $\I^{\bu}$ is a right $B$-module. We observe from (\ref{action1}),
(\ref{action2}) and (\ref{action3}) that $\I^{\bu}$ is a right graded $B$-module.

The following result states that $\I^{\bu}$ is a dg $A$-$B$-bimodule. We mention that it is evident that $\I^{\bu}$ is a graded $A$-$B$-bimodule.
Recall that the differentials on $\I^{\bu}$ are denoted by $\p^l$.

\begin{prop} \label{partialaction} For each $i\in Q_{0}$, $l\in\Z$ and $(p, q)\in \mathbf{B}^{l}_{i}$, let $x\zeta_{(p, q)}\in I_i\zeta_{(p, q)}$. Then for each vertex $j\in Q_{0}$ and each arrow $\beta\in Q_{1}$, we have
\begin{enumerate}
\item[(1)] $\partial^{l}(x\zeta_{(p, q)}\cdot e_{j})=\partial^{l}(x\zeta_{(p, q)})\cdot e_{j};$

\item[(2)] $\partial^{l+1}(x\zeta_{(p, q)}\cdot \beta)=\partial^{l}(x\zeta_{(p, q)})\cdot \beta;$

\item[(3)] $\partial^{l-1}(x\zeta_{(p, q)}\cdot\beta^{*})=\partial^{l}(x
\zeta_{(p, q)})\cdot \beta^{*}.$
\end{enumerate} In other words, the right $B$-action makes $\mathcal{I}^{\bullet}$ a
right dg $B$-module and thus a dg $A$-$B$-bimodule.\hfill $\square$
\end{prop}

We make some preparation for the proof of Proposition \ref{partialaction}.
There is a unique right $B$-module morphism $\psi: B\xra \I^{\bullet}$ with $\psi(1)=\sum_{i\in Q_0}e_i^{\s}\z_{(e_i, e_i)}$. Here, $1$ is the unit of $B$.
For each arrow $\beta\in Q_{1}$, there is a unique right $B$-module morphism $\psi_{\b}: B\xra \I^{\bullet}$ with $\psi_{\b}(1)=\b^{\s}\z_{(e_{t(\b)}, e_{t(\b)})}$.
Then by Lemma \ref{fact} we have \begin{equation}\label{psipsi}\psi(p^*q)=e_{s(q)}^{\s}\z_{(p, q)}
\text{\qquad and\qquad} \psi_{\b}(p^*q)=\delta_{s(q), t(\b)}\b^{\s}\z_{(p, q)}\end{equation} for $p^*q\in\La$.
Here, $\La$ is the basis of $B=L_k(Q)^{\rm op}$ given in Lemma \ref{lbasis}.
It follows that $\psi$ is injective.

We have the following observation.

\begin{lem} \label{partialmodule}
Let $(p, q)$ be an admissible pair in $Q$ with $l(q)-l(p)=l$ for some $l\in\Z$.
Then for each $\b\in Q_1$, we have that $(\p^{l}\circ\psi_{\b})(p^{*}q)=\psi(p^{*}q\b)$.
From this, we conclude that $(\p^l\c \psi_{\b})(a)=\psi(a\b)$ for any element $a\in B^l$.
\end{lem}

We observe that both $\psi$ and $\psi_{\b}$ are graded $B$-module morphisms. In particular,
we have $\psi_{\b}(p^*q)\in \I^l$ and $\psi(p^*q\b)\in \I^{l+1}$. Here,
$p^*q\b$ is the multiplication of $p^*q$ and $\b$ in $L_k(Q)$.

\begin{proof} We first make an observation.
For each $\b\in Q_1$ and $p^*q\in\La$, we have that
\begin{equation}
\label{eq:o}
p^*q\b=
\begin{cases}
(\widehat{p})^*-\sum_{\g\in S(\b)}{(\widehat{p})}^{*}\g^*\g,& \text{if $l(q)=0$, $p=\b \widehat{p}$ and $\b$ is special};\\
\delta_{s(q), t(\b)}p^*q\b,& \text{otherwise}.
\end{cases}
\end{equation}
Then we have that
\begin{align*}
(\p^{l}\circ\psi_{\b})(p^{*}q)
&=\delta_{s(q), t(\b)}\p^l(\b^{\s}\z_{(p, q)})\\
&=\left\{\begin{array}{ll}
e_{s(\b)}^{\s}\z_{(\widehat{p}, e_{s(\b)})}-\sum\limits_{\g\in S(\b)}e_{s(\b)}^{\s}\z_{(\g\widehat{p}, \g)},& \begin{matrix}
\text{if~~} l(q)=0, ~~p=\b \widehat{p}\\ \text{and~~}\b \text{~~is ~~special};
\end{matrix}\\
\delta_{s(q), t(\b)}e_{s(\b)}^{\s}\z_{(p, q\b)},& \text{otherwise}.
 \end{array}\right.\\
&=\psi(p^{*}q\b)
\end{align*} Here, the second equality uses Definition \ref{definj} of the differential $\p^l$ and the last equality uses (\ref{eq:o}).
\end{proof}

\begin{comment}\begin{lem} \label{partialaction} For each $i\in Q_{0}$, $l\in\Z$ and $(p, q)\in \mathbf{B}^{l}_{i}$, let $x\zeta_{(p, q)}\in I_i\zeta_{(p, q)}$. Then we have that for each vertex $j\in Q_{0}$ and each arrow $\beta\in Q_{1}$
\begin{enumerate}
\item[(1)] $\partial^{l}(x\zeta_{(p, q)}\cdot e_{j})=\partial^{l}(x\zeta_{(p, q)})\cdot e_{j};$

\item[(2)] $\partial^{l+1}(x\zeta_{(p, q)}\cdot \beta)=\partial^{l}(x\zeta_{(p, q)})\cdot \beta;$

\item[(3)] $\partial^{l-1}(x\zeta_{(p, q)}\cdot\beta^{*})=\partial^{l}(x
\zeta_{(p, q)})\cdot \beta^{*}.$
\end{enumerate}
\end{lem}
\end{comment}

\vskip 10pt

\noindent{\emph{Proof of Proposition \ref{partialaction}.}}
We observe that $\p^l(e_i^{\s}\zeta_{(p, q)})=0$ for each $l\in\Z$.
It follows immediately that $(1)$, $(2)$ and $(3)$ hold for $x=e_i^{\s}$.
Then it suffices to prove $(1)$, $(2)$ and $(3)$ for $x=\aa^{\s}$ with $\aa\in Q_1$.
Recall that $(p, q)\in \Bb^l_i$, and thus $t(\aa)=i$.

For (1), we have that
\begin{equation*}
\begin{split}
\partial^{l}(\alpha^{\s}\zeta_{(p, q)}\cdot e_{j})
&=\partial^{l}(\psi_{\aa}(p^*q)\cdot e_{j})
=(\p^{l}\circ\psi_{\aa})(e_jp^*q)\\
&=\psi(e_{j}p^*q\aa)
=\psi(p^*q\aa)\cdot e_{j}\\
&=\partial^{l}(\alpha^{\s}\zeta_{(p, q)})\cdot e_{j}.\\
\end{split}
\end{equation*} Here, the second and the fourth equalities hold because $\psi_{\aa}$ and $\psi$ are right $B$-module morphisms; the third and the last equalities use Lemma
\ref{partialmodule}. Similar arguments prove $(2)$ and $(3)$.
\hfill $\square$

\section{The differential graded endomorphism algebra of the injective Leavitt complex}
\label{sfifth}

In this section, we prove that the differential graded endomorphism
algebra of the injective Leavitt complex is quasi-isomorphic to the Leavitt path algebra. Here, the Leavitt path algebra is naturally $\Z$-graded and viewed as a dg algebra with trivial differential.

\subsection{The quasi-balanced dg bimodule}
\label{subsection51}

\begin{comment}For a graded algebra $A$,
we denote by $A\-\Gr$ the abelian category of (left) graded $A$-modules with
degree-preserving morphisms.
To a graded $A$-module, we associate a
graded $A$-module $M(1)$ as follows: M(1)=M as ungraded modules with
the grading given by $M(1)^{n}=M^{n+1}$.
This gives rise to an automorphism $(1)$:$A\-\Gr\xra A\-\Gr$
of categories. Similarly, there is an automorphism $(-1)$:$A\-\Gr\xra A\-\Gr$.
Iteratively, for $n\in \Z$ there is an automorphism $(n)$:$A\-\Gr\xra A\-\Gr$.
We denote by $\Gr\-A$ the abelian category of graded right $A$-modules with
degree-preserving morphisms. Similarly, an automorphism $(n)$:$\Gr\-A\xra \Gr\-A$ can be defined for $n\in \Z$.
\end{comment}

We first recall some notation on quasi-balanced dg bimodules.
Let $A$ be a dg algebra and $M, N$ be (left) dg $A$-modules.
We have a $\Z$-graded vector space ${\rm Hom}_{A}(M, N)=\bigoplus_{n\in\Z}{\rm Hom}_{A}(M, N)^{n}$
such that each component ${\rm Hom}_{A}(M, N)^{n}$ consists of $k$-linear maps
$f:M\xra N$ satisfying $f(M^i)\subseteq N^{i+n}$ for all $i\in\Z$ and $f(a\cdot m)=(-1)^{n|a|}a
\cdot f(m)$ for all homogenous elements $a\in A$.
The differential on ${\rm Hom}_{A}(M, N)$ sends $f\in
{\rm Hom}_{A}(M, N)^{n}$ to $d_{N}\circ f-(-1)^n f\circ d_{M}\in
{\rm Hom}_{A}(M, N)^{n+1}$.
Furthermore, ${\rm End}_{A}(M):={\rm Hom}_{A}(M, M)$ becomes a dg algebra
with this differential and the usual composition as multiplication.
The dg algebra ${\rm End}_{A}(M)$ is usually called the \emph{dg endomorphism algebra} of $M$.

We denote by $A^{\rm opp}$ the \emph{opposite dg algebra} of a dg algebra $A$. More precisely,
$A^{\rm opp}=A$ as graded spaces with the same differential, and the multiplication $``\circ"$
on $A^{\rm opp}$ is given by $a\circ b=(-1)^{|a||b|}ba$.

Let $B$ be another dg algebra.
Recall that a right dg $B$-module is a left dg $B^{\rm opp}$-module.
For a dg $A$-$B$-bimodule $M$, the canonical map
$A\xra {\rm End}_{B^{\rm opp}}(M)$ is a homomorphism of dg algebras, sending $a$ to $l_{a}$ with $ l_{a}(m)=a\cdot m$ for $a\in A$ and $m\in M$.
Similarly, the canonical map
$B\xra {\rm End}_{A}(M)^{\rm opp}$ is a homomorphism of dg algebras,
sending $b$ to $ r_{b}$ with $ r_{b}(m)=(-1)^{|b||m|}m\cdot b$ for homogenous elements $b\in B$ and $m\in M$.

A dg $A$-$B$-bimodule $M$ is called \emph{right quasi-balanced} provided that the canonical
homomorphism $B \xra {\rm End}_{A}(M)^{\rm opp}$ of dg algebras
is a quasi-isomorphism; see \cite[2.2]{cy}.

Denote by $\K(A)$ the homotopy category and by $\D(A)$ the derived category
of left dg $A$-modules; they are triangulated categories with arbitrary coproducts.
For a dg $A$-$B$-bimodule
$M$ and a left dg $A$-module $N$, ${\rm Hom}_{A}(M, N)$ has a natural structure of left dg $B$-module with the graded $B$-module structure given by the formula $(bf)(m) = (-1)^{|b|(|f|+|m|)}f(mb)$, where $b\in B$, $f\in {\rm Hom}_{A}(M, N)$ and $m\in M$.

The following lemma is \cite[Proposition 2.2]{cy}; compare \cite[4.3]{ke} and \cite[Appendix A]{kr}.

\begin{lem} \label{equiv} Let $M$ be a dg $A$-$B$-bimodule which is right quasi-balanced.
Recall that ${\rm Loc}\langle M\rangle\subseteq\K(A)$ is the smallest triangulated
subcategory of
$\K(A)$ which contains $M$ and is closed under arbitrary coproducts. Assume that
$M$ is a compact object in ${\rm Loc}\langle M\rangle$.
Then we have a triangle equivalence \begin{equation*}{\rm Hom}_{A}(M,-):{\rm Loc}\langle M\rangle \stackrel{\sim}\longrightarrow \D(B).\tag*{$\square$}\end{equation*}
\end{lem}

In what follows, $Q$ is a finite quiver without sinks and $A=kQ/J^{2}$ is the corresponding algebra with radical square zero.
Consider $A$ as a dg algebra concentrated on degree zero.
Recall that the opposite Leavitt path algebra $B=L_k(Q)^{\rm op}$ is naturally $\Z$-graded, and that it is viewed as
a dg algebra with trivial differential.

Recall from Proposition \ref{partialaction} that the injective Leavitt complex $\I^{\bu}$ is a dg $A$-$B$-bimodule. The following result establishes a connection between the injective Leavitt complex and the Leavitt path algebra, which justifies the terminology.

\begin{thm} \label{rightq} Let $Q$ be a finite quiver without sinks.  Then the dg $A$-$B$-bimodule
$\mathcal{I}^{\bullet}$ is right quasi-balanced.

In particular, the dg endomorphism algebra ${\rm End}_{A}(\I^{\bu})$ is quasi-isomorphic to the Leavitt path algebra $L_k(Q)$. Here, $L_k(Q)$ is naturally $\Z$-graded and viewed as a dg algebra with trivial differential. \hfill $\square$
\end{thm}

We will prove Theorem \ref{rightq} in subsection \ref{subsection52}.
The following triangle equivalence has been proved by \cite[Theorem 6.1]{cy}.

\begin{thm}\label{coreq} Let $Q$ be a finite quiver without sinks. Then there is a triangle
equivalence $${\rm Hom}_{A}(\I^{\bullet},-):\K_{\rm ac}(A\-\Inj) \stackrel{\sim}\longrightarrow\D(B),$$
which sends $\I^{\bu}$ to $B$ in $\D(B)$.
\end{thm}

\begin{proof} By Theorem \ref{tc} $\I^{\bullet}$ is a compact object in $\K_{\rm ac}(A\-\Inj)={\rm Loc}\langle \I^{\bullet}\rangle$.
Then the triangle equivalence follows from Theorem \ref{rightq} and Lemma \ref{equiv}.
The canonical map $B\xra {\rm End}_{A}(\mathcal{I}^{\bullet})^{\rm opp}$, which is a
quasi-isomorphism, identifies ${\rm Hom}_{A}(\I^{\bullet}, \I^{\bullet})$ with $B$ in $\D(B)$.
\end{proof}

\begin{rmk} 
\begin{enumerate} \item[(1)] We observe that the injective Leavitt complex $\I^{\bu}$ can be also obtained by applying a quasi-inverse of the equivalence in \cite[Theorem 6.1]{cy} to the regular $B$-module $B$, since $B\in \D(B)$ is a compact generator.

\item[(2)] The assumption that $Q$ has no sinks is necessary for Theorem \ref{coreq}. Indeed, if $Q$ has only one vertex without arrows, we have $A\cong B\cong k$. In this case, the category $\K_{\rm ac}(A\-\Inj)$ is trivial while $\D(B)$ is nontrivial.
\end{enumerate}
\end{rmk}

\subsection{The proof of Theorem \ref{rightq}}
\label{subsection52}

We follow the notation in subsection \ref{subsection51}.

\begin{lem} $($\cite[Theorem 4.8]{t}$)$ \label{uniquethm}
Let $A$ be a $\Z$-graded algebra and $\varphi:L_k(Q)\xra A$ be a graded algebra homomorphism with
$\varphi(e_{i})\neq 0$ for all $i\in Q_{0}$. Then $\varphi$ is injective.\hfill $\square$
\end{lem}

Let $Z^{n}$ and $C^{n}$ denote the $n$-th cocycle and
coboundary of the dg endomorphism algebra ${\rm End}_{A}(\mathcal{I}^{\bullet})$, respectively.

\begin{lem} \label{coboundary} Any element $f:\mathcal{I}^{\bullet}\xra \mathcal{I}^{\bullet}$ in $C^{n}$
satisfies that $f(\mathcal{I}^{l})\subseteq \Ke\partial^{n+l}$ for each integer $l$.
\end{lem}

\begin{proof} We write $f=(f^l)_{l\in \Z}$. Since $f\in C^{n}$, there exists
$h=(h^l)_{l\in\Z}\in {\rm End}_{A}(\mathcal{I}^{\bullet})^{n-1}$ such that $f^{l}=\partial^{n+l-1} \circ h^l-(-1)^{n-1}h^{l+1}\circ \partial^{l}$ for any $l\in \Z$.
Because of $\Im \p^{n+l-1}\subseteq \Ke \p^{n+l}$ and $\Im \p^{l}\subseteq \Ke \p^{l+1}$,
it suffices to show that
$h^{l+1}(\Ke\partial^{l+1})\subseteq \Ke\partial^{n+l}$. By Lemma \ref{partialker} $\{e_{i}^{\s}\zeta_{(p, q)}\;|\; i\in Q_0, (p, q)\in \Bb^{l+1}_i\}$ is a $k$-basis of $\Ke\p^{l+1}$. Then we are done by
Lemma \ref{a}.
\end{proof}

We denote by
$\rho: B\xra {\rm End}_{A}(\mathcal{I}^{\bullet})^{\rm opp}$ the canonical map
sending $b$ to $\rho(b)$ with $\rho(b)(m)=(-1)^{|b||m|}m\cdot b$ for homogenous elements $b\in B$ and $m\in \I^{\bu}$.
Since $B$ is a dg algebra with trivial differential,
we have that $\rho(B^{n})\subseteq Z^{n}$ for each $n\in \Z$.
Taking cohomologies, we have the graded algebra homomorphism \begin{equation}
\label{coho}
H(\rho):
B\longrightarrow H({\rm End}_{A}(\mathcal{I}^{\bullet})^{\rm opp}).\end{equation}

\begin{prop} \label{embedding} The graded algebra homomorphism $H(\rho)$ is injective.
\end{prop}

\begin{proof} Recall $B=L_k(Q)^{\rm op}$. By Lemma \ref{uniquethm}, it suffices to prove that $H(\rho)(e_i)\neq 0$, or equivalently,
$\rho(e_i)\notin C^0$ for all $i\in Q_0$.
Since $Q$ has no sinks, for each $i\in Q_0$ there exists an arrow $\aa$ in $Q$ such that $s(\aa)=i$. Observe that $\aa^{\s}\z_{(\aa, e_{t(\aa)})}\in \I^{-1}$.
Then $\rho(e_i)(\aa^{\s}\zeta_{(\aa, e_{t(\aa)})})=\aa^{\s}\zeta_{(\aa, e_{t(\aa)})}\cdot e_i=\aa^{\s}\zeta_{(\aa, e_{t(\aa)})}$. By Lemma \ref{partialker}
we have $\aa^{\s}\zeta_{(\aa, e_{t(\aa)})}\notin \Ke\p^{-1}$.
It follows from Lemma \ref{coboundary} that $\rho(e_i)\notin C^0$.
\end{proof}

We will prove that the graded algebra homomorphism $H(\rho)$ is surjective.
For each $y\in Z^n$, we will find an element $x\in B^n$ with $y-\rho(x)\in C^n$.
The argument is quite involved.

In what follows, we fix $y\in Z^{n}$ for some $n\in \Z$.
Then $y:\I^{\bu}\xra \I^{\bu}$ satisfies $\p^{\bullet} \circ y-(-1)^{n}y\circ \partial^{\bullet}=0$. Recall that $\I^{l}=\bigoplus_{i\in Q_0}I_i^{(\Bb^l_i)}$ for each $l\in \Z$. The set
$\{e_i^{\s}\z_{(p, q)}, \aa^{\s}\z_{(p, q)}\;|\;i\in Q_0, (p, q)\in \Bb^l_i \;\text{and}\; \aa\in Q_1 \;\text{with} \;t(\aa)=i\}$ is a $k$-basis of $\I^{l}$. For each $i\in Q_{0}$ and $(p, q)\in\Bb^l_i$,
we have that
\begin{equation}
\label{eq:l'}
\begin{cases}
(\partial^{n+l}\circ y)(e_{i}^{\s}\zeta_{(p, q)})=0 & \\
(\partial^{n+l}\circ y)(\alpha^{\s}\zeta_{(p, q)})=(-1)^{n}(y\circ \partial^{l})(\alpha^{\s}\zeta_{(p, q)}),&
\end{cases}
\end{equation} where $\alpha \in Q_{1}$ with $t(\aa)=i$.

Observe that $y$ is an $A$-module morphism. By Lemma \ref{a} we may assume that
\begin{equation}\label{eq:par}
\begin{cases}
y(e_{i}^{\s}\zeta_{(p, q)})=\psi(\nu_{(p, q)})\\
y(\alpha^{\s}\zeta_{(p, q)})=\psi(\mu^{\alpha}_{(p, q)})
+\psi_{\alpha}(\nu_{(p, q)})
\end{cases}
\end{equation} for some $\nu_{(p, q)}\in L_k(Q)^{n+l}e_i$ and
$\mu^{\alpha}_{(p, q)}\in L_k(Q)^{n+l}e_{s(\aa)}$. We refer to $(\ref{psipsi})$ for the
right $B$-module morphisms $\psi$ and $\psi_{\aa}$.

By (\ref{eq:par}) and Lemma \ref{partialmodule} we have that \begin{equation*}(\partial^{n+l}\circ y)(\alpha^{\s}\zeta_{(p, q)})
=(\partial^{n+l}\c\psi_{\alpha})(\nu_{(p, q)})=\psi(\nu_{(p, q)}\aa),\end{equation*} and that
\begin{equation*}(y\circ \partial^{l})(\alpha^{\s}\zeta_{(p, q)})
=y((\partial^{l}\c\psi_{\alpha})(p^*q))=y(\psi(p^*q\aa)).\end{equation*}
Here, we implicitly use the fact that $\p^{\bu}\c \psi=0$.
By $(\ref{eq:l'})$, we obtain the equality $\psi(\nu_{(p, q)}\alpha)
=(-1)^{n}y(\psi(p^*q\aa))$.
Recall that $\psi$ is injective. By (\ref{eq:o}) and $(\ref{eq:par})$, we infer that \begin{equation}
\label{eq:l''}
\begin{split}
\nu_{(p, q)}\alpha
&=\begin{cases}
(-1)^{n}(\nu_{(\widehat{p}, e_{s(\alpha)})}-\sum\limits_{\beta\in S(\alpha)} \nu_{(\beta\widehat{p}, \beta)}), &
\begin{matrix}\text{if~} q=e_{i}, ~~p=\alpha \widehat{p}\\ \text{~and~} \alpha \text{~is special};\end{matrix}\\
(-1)^{n}\nu_{(p, q\alpha)},& \text{otherwise}.
\end{cases}
\end{split}
\end{equation} Here, we recall from $(\ref{eq:vv})$ the definition of $S(\aa)$.

\begin{lem} \label{rule} Keep the notation as above. Take                                                                                                                                                     $x=\sum_{j\in Q_{0}}\nu_{(e_{j}, e_{j})}\in B^n=L_k(Q)^{n}$.
Then we have $\nu_{(p, q)}=(-1)^{nl}xp^{*}q$ in $L_k(Q)$ for each $(p, q)\in\Bb^l_i$.
\end{lem}

\begin{proof} We clearly have $\nu_{(e_{i}, e_{i})}=xe_{i}$ in $L_k(Q)$.
If $l(p)=0$ and $l(q)>0$, by $(\ref{eq:l''})$ we have
$\nu_{(e_{t(q)},q)}=(-1)^{n}
\nu_{(e_{t(q)},\widetilde{q})}\b$ with $q=\widetilde{q}\b$
and $\widetilde{q}$ the truncation of $q$.
By induction on $l(q)$, we obtain that $\nu_{(e_{t(q)},q)}=(-1)^{nl(q)}xq$ in $L_k(Q)$.

If $l(p)>0$, we write $p=\aa \widehat{p}$ in $Q$ with $\aa\in Q_1$.
Let $\beta\in Q_{1}$ with $s(\b)=s(\aa)$. If $\b$ is not special, by $(\ref{eq:l''})$ we have
$\nu_{(\beta\widehat{p}, e_{t(\beta)})}\beta
=(-1)^{n}\nu_{(\b\widehat{p}, \b)}$.
If $\b$ is special, by $(\ref{eq:l''})$ we have
$\nu_{(\beta\widehat{p}, e_{t(\beta)})}\beta
=(-1)^{n}(\nu_{(\widehat{p}, e_{t(\widehat{p})})}-\sum_{\g\in S(\b)}\nu_{(\g\widehat{p}, \g)})$. We
put the two cases together to obtain the following identity $$\sum\limits_{\{\beta\in Q_{1}\;|\;s(\beta)=s(\aa)\}}\nu_{(\beta\widehat{p}, e_{t(\beta)})}\beta
=(-1)^{n}\nu_{(\widehat{p}, e_{t(\widehat{p})})}.$$ Multiplying $\aa^{*}$ from the right and using the relation $(3)$ in Definition \ref{defleavitt}, we obtain
$\nu_{(p, e_{t(p)})}=(-1)^{n} \nu_{(\widehat{p}, e_{t(\widehat{p})})}\aa^{*}$.
By induction on $l(p)$, we have $\nu_{(p, e_{t(p)})}=(-1)^{nl(p)}xp^*$ in $L_k(Q)$.
This proves the case with $l(p)>0$ and $l(q)=0$.

If $l(p)>0$ and $l(q)>0$, we have $\nu_{(p, q)}=(-1)^{n l(q)}\nu_{(p, e_{t(p)})}q$ by iterating $(\ref{eq:l''})$. Then we are done by $\nu_{(p, e_{t(p)})}=(-1)^{n l(p)}xp^*$.
\end{proof}

We will construct a map $h:\I^{\bu}\xra \I^{\bu}$ of degree $n-1$,
which will prove that $y-\rho(x)\in C^n$; see Lemma \ref{last}.
For the construction of $h$, we assign to each admissible pair
$(p, q)$ an element $\o_{(p, q)}$ in $L_k(Q)$.

For each vertex $i\in Q_0$, set $\o_{(e_{i}, e_{i})}=\sum_{\{\beta\in Q_{1}\;|\; s(\beta)=i\}}\mu^{\beta}_{(\beta, e_{t(\beta)})}\in L_k(Q)^{n-1}e_i$.
Here, we recall from (\ref{eq:par}) the element $\mu^{\beta}_{(\beta, e_{t(\beta)})}$.
We define $\o_{(p, e_{t(p)})}$ inductively by \begin{equation}\label{admpair1}\o_{(p, e_{t(p)})}=(-1)^{n-1}\o_{( \widehat{p}, e_{t(\widehat{p})})}
\g^{*}+\sum\limits_{\{\beta\in Q_{1}\;|\; s(\b)=s(\g)\}
}\mu^{\b}_{(\b \widehat{p}, e_{t(\b)})}\g^{*}\in L_k(Q)^{n+l-1}e_{t(p)},
\end{equation} where the path
$p =\g \widehat{p}$ with $\g\in Q_1$ is of length $-l\geq 1$.
Let $(p, q)$ be an admissible pair in $Q$ with $l(q)>0$.
We define $\o_{(p, q)}$ by induction on the length of $q$ as follows \begin{equation}\label{admpair22}\o_{(p, q)}=(-1)^{n-1}(\o_{(p, \widetilde{q})}\d-\mu^
{\d}_{(p,\widetilde{q})})\in L_k(Q)^{n+l-1}e_{s(q)},
\end{equation} where the path $q=\widetilde{q}\d$ with $\d\in Q_1$ is of length $l+l(p)$.

\begin{lem}\label{lateruse} Let $(p, q)$ be an admissible pair in $Q$. For $\aa\in Q_1$ with $t(\aa)=s(q)$, we have
\begin{equation*}
\o_{(p, q)}\alpha-\mu^{\alpha}_{(p, q)}
=\begin{cases}
(-1)^{n-1}(\o_{(\widehat{p}, e_{t(\widehat{p})})}-\sum\limits_{\beta\in S(\alpha)}\o_{(\beta\widehat{p}, \beta)}), &
\begin{matrix}\text{if~} l(q)=0, ~~p=\alpha \widehat{p}\\ \text{~and~} \alpha \text{~is special};\end{matrix}\\
(-1)^{n-1}\o_{(p, q\alpha)},& \text{otherwise}.
\end{cases}
\end{equation*}
\end{lem}

\begin{proof} In the first case, we have \begin{equation*}
\begin{split}
&\o_{(\widehat{p}, e_{t(\widehat{p})})}-\sum\limits_{\b\in S(\aa)}\o_{(\b \widehat{p}, \b)}\\
=&\o_{(\widehat{p}, e_{t(\widehat{p})})}-(-1)^{n-1}\sum\limits_{\b\in S(\aa)}(\o_{(\b \widehat{p}, e_{t(\b)})}\b-\mu^{\b}_{(\b \widehat{p}, e_{t(\b)})})\\
=&\o_{(\widehat{p}, e_{t(\widehat{p})})}-\sum\limits_{\b\in S(\aa)}\o_{(\widehat{p}, e_{t(\widehat{p})})}\b^*\b\\
&-(-1)^{n-1}\sum\limits_{\b\in S(\aa)}(\sum\limits_{\{\g\in Q_1\;|\;s(\g)=s(\b)\}}\mu^{\g}_{(\g \widehat{p}, e_{t(\g)})}\b^*\b-\mu^{\b}_{(\b \widehat{p}, e_{t(\b)})})\\
=&\o_{(\widehat{p}, e_{t(\widehat{p})})}\aa^*\aa+(-1)^{n-1}(\sum\limits_{\{\b\in Q_1\;|\;s(\b)=s(\aa)\}}\mu^{\b}_{(\b\widehat{p}, e_{t(\b)})}\aa^*\aa-\mu^{\aa}_{(p, q)})\\
=&(-1)^{n-1}(\o_{(\aa \widehat{p}, e_{t(\aa)})}\aa-\mu^{\aa}_{(\aa \widehat{p}, e_{t(\aa)})}).
\end{split}
\end{equation*} Here, the first equality uses $(\ref{admpair22})$,
the second and the last equalities use $(\ref{admpair1})$
for the admissible pairs
$(\b\widehat{p}, e_{t(\b)})$ and $(\aa \widehat{p}, e_{t(\aa)})$, respectively,
and the third equality uses the relation $(4)$ in Definition \ref{defleavitt}.
The second case follows directly from $(\ref{admpair22})$.
\end{proof}

We define a $k$-linear map $h:\mathcal{I}^{\bullet}\xra\I^{\bu}$ such that
\begin{equation}\label{hdef}
h(e_{i}^{\s}\zeta_{(p, q)})=
\psi(\omega_{(p, q)})
\text{\qquad and \qquad}h(\alpha^{\s}\zeta_{(p, q)})
=\psi_{\alpha}(\o_{(p, q)})
\end{equation} for each $i\in Q_{0}$, $l\in\Z$, $(p, q)\in\Bb^l_i$ and $\alpha\in Q_{1}$ with $t(\aa)=i$.
By Lemma \ref{a} the map $h$ is a left $A$-module morphism.
We observe that $h$ is of degree $n-1$, and thus $h\in \End_A(\I^{\bu})^{n-1}$.

\begin{lem}\label{last} Let $x$ be the element in Lemma \ref{rule} and $h$ be the morphism given by $(\ref{hdef})$. For each $i\in Q_{0}$, $l\in\Z$ and $(p, q)\in\Bb^l_i$, we have that
\begin{equation*}
\begin{cases}(y-\rho(x))(e_{i}^{\s}\zeta_{(p, q)})
=(\p^{n+l-1}\c h-(-1)^{n-1}h\c \p^l)(e_{i}^{\s}\zeta_{(p, q)})=0\\
(y-\rho(x))(\alpha^{\s}\zeta_{(p, q)})
=(\p^{n+l-1} \circ h-(-1)^{n-1}h
\circ \p^l)(\alpha^{\s}\zeta_{(p, q)}),
\end{cases}
\end{equation*} where $\aa$ is an arrow with $t(\aa)=i$.
\end{lem}

\begin{proof} Recall from $(\ref{psipsi})$ the right $B$-module morphisms $\psi$ and $\psi_{\b}$ for $\b\in Q_1$. Because of $\p^l\c \psi=0$ for $l\in\Z$, we have $(\p^{n+l-1}\c h-(-1)^{n-1}h\c \p^l)(e_{i}^{\s}\zeta_{(p, q)})=0$.
By (\ref{eq:par}) and Lemma \ref{rule}, we have
\begin{equation*}
\begin{cases}y(e_{i}^{\s}\zeta_{(p, q)})
=(-1)^{nl}\psi(xp^{*}q)=\rho(x)(e_{i}^{\s}\zeta_{(p, q)})\\
\psi_{\aa}(\nu_{(p, q)})=\rho(x)(\aa^{\s}\zeta_{(p, q)}).
\end{cases}
\end{equation*} It follows that
$$(y-\rho(x))(e_{i}^{\s}\zeta_{(p, q)})=0
=(\p^{n+l-1}\c h-(-1)^{n-1}h\c \p^l)(e_{i}^{\s}\zeta_{(p, q)}),$$
and by (\ref{eq:par}) we have
$(y-\rho(x))(\alpha^{\s}\zeta_{(p, q)})
=\psi(\mu^{\alpha}_{(p, q)})$.

It remains to prove $(\p^{n+l-1}\circ h-(-1)^{n-1}h
\circ \p^l)(\alpha^{\s}\zeta_{(p, q)})=\psi(\mu^{\alpha}_{(p, q)})$.
We have \begin{equation*}
\begin{split}&(\p^{n+l-1}\circ h-(-1)^{n-1}h
\circ \p^l)(\alpha^{\s}\zeta_{(p, q)})\\
&=(\p^{n+l-1} \circ \psi_{\aa})(\o_{(p, q)})-(-1)^{n-1}h((\p^l\c\psi_{\aa})(p^*q))\\
&=\psi(\o_{(p, q)}\aa)-(-1)^{n-1}h(\psi(p^*q\aa))\\
&=\psi(\o_{(p, q)}\aa)-\psi(\o_{(p, q)}\aa-\mu^{\aa}_{(p, q)})\\
&=\psi(\mu^{\aa}_{(p, q)}).
\end{split}
\end{equation*} Here, the second equality uses Lemma \ref{partialmodule} and the third
equality uses $(\ref{eq:o})$ and Lemma \ref{lateruse}.
\end{proof}

\begin{prop} \label{surjprop}The graded algebra homomorphism $H(\rho)$ in (\ref{coho}) is surjective. 
\end{prop}

\begin{proof} For any $n\in\Z$, we need to prove that $H^{n}(\rho)$ is surjective. For any element $\overline{y}=y+C^{n}$ with $y\in Z^{n}$,
take $x=\sum_{i\in Q_{0}}\nu_{(e_{i}, e_{i})}\in B^n=L_k(Q)^{n}$. By Lemma \ref{last}, we have $y-\rho(x)\in C^{n}$. Then it follows that $\overline{y}=\overline{\rho(x)}$ in $H^n(\End_A(\I^{\bu})^{\rm opp})$.
\end{proof}

\vskip 10pt

\noindent{\emph{Proof of Theorem \ref{rightq}.}}  By Proposition \ref{embedding} and Proposition \ref{surjprop}, the graded algebra homomorphism $H(\rho): B\xra H({\rm End}_{A}(\mathcal{I}^{\bullet})^{\rm opp})$ in (\ref{coho}) is an isomorphism. Recall that $B=L_k(Q)^{\rm op}$ is a dg algebra with trivial differential. There is a graded algebra isomorphism $B^{\rm op}\xrightarrow[]{\sim}B^{\rm opp}$ sending $b$ to $(-1)^{\frac{n(n+1)}{2}}b$ for $b\in B^n$, where $B^{\rm op}$ is the usual opposite algebra of $B$. Then the dg endomorphism algebra $\End_A(\I^{\bu})$ is quasi-isomorphic to
$B^{\rm opp}$, and also to the Leavitt path algebra $L_k(Q)=B^{\rm op}$.
\hfill $\square$

\section{Acknowledgements} The author thanks her supervisor Professor Xiao-Wu Chen for inspiring discussions and encouragement. This project was supported by the National Natural Science Foundation of China (No.s 11522113 and 11571329). The author also gratefully acknowledges the support of Australian Research Council grant DP160101481.

\vskip 10pt

{\footnotesize\noindent Huanhuan ${\rm Li}^{1}$ \\
1. Centre for Research in Mathematics, Western Sydney University, Australia.\\
E-mail: h.li@westernsydney.edu.au}

\end{document}